\documentclass[11pt,english]{article}
\usepackage[T1]{fontenc}
\usepackage[latin9]{inputenc}
\usepackage[a4paper]{geometry}
\usepackage{amsthm}
\usepackage{amsmath}
\usepackage{amssymb}
\usepackage{esint}

\makeatletter
\theoremstyle{plain}
\newtheorem{thm}{\protect\theoremname}[section]
  \theoremstyle{plain}
  \newtheorem{cor}[thm]{\protect\corollaryname}
  \theoremstyle{plain}
  \newtheorem{question}[thm]{\protect\questionname}
  \theoremstyle{plain}
  \newtheorem*{thm*}{\protect\theoremname}
  \theoremstyle{definition}
  \newtheorem{defn}[thm]{\protect\definitionname}
  \theoremstyle{remark}
  \newtheorem{notation}[thm]{\protect\notationname}
  \theoremstyle{plain}
  \newtheorem{lem}[thm]{\protect\lemmaname}
  \theoremstyle{plain}
  \newtheorem{prop}[thm]{\protect\propositionname}
  \theoremstyle{remark}
  \newtheorem*{rem*}{\protect\remarkname}
\newcommand{\lyxaddress}[1]{
\par {\raggedright #1
\vspace{1.4em}
\noindent\par}
}

\DeclareMathOperator{\sdim}{s-dim}

\DeclareMathOperator{\var}{Var}

\DeclareMathOperator{\bdim}{bdim}
\DeclareMathOperator{\pdim}{pdim}

\date{}

\makeatletter
\def\blfootnote{\xdef\@thefnmark{}\@footnotetext}
\makeatother

\makeatother

\usepackage{babel}
  \providecommand{\corollaryname}{Corollary}
  \providecommand{\definitionname}{Definition}
  \providecommand{\lemmaname}{Lemma}
  \providecommand{\notationname}{Notation}
  \providecommand{\propositionname}{Proposition}
  \providecommand{\questionname}{Question}
  \providecommand{\remarkname}{Remark}
  \providecommand{\theoremname}{Theorem}
\providecommand{\theoremname}{Theorem}

\begin{document}

\title{On self-similar sets with overlaps and\\
 inverse theorems for entropy}

\author{Michael Hochman}
\maketitle
\begin{abstract}
We\blfootnote{Supported by ERC grant 306494}\blfootnote{\emph{2010 Mathematics Subject Classication}. 28A80, 11K55, 11B30, 11P70}
study the Hausdorff dimension of self-similar sets and measures on
$\mathbb{R}$. We show that if the dimension is smaller than the minimum
of 1 and the similarity dimension, then at small scales there are
super-exponentially close cylinders. This is a step towards the folklore
conjecture that such a drop in dimension is explained only by exact
overlaps, and confirms the conjecture in cases where the contraction
parameters are algebraic. It also gives an affirmative answer to a
conjecture of Furstenberg, showing that the projections of the ``1-dimensional
Sierpinski gasket'' in irrational directions are all of dimension
1.

As another consequence, when a family of self-similar sets or measures
is parametrized in a real-analytic manner, then, under an extremely
mild non-degeneracy condition, the set of ``exceptional'' parameters
has Hausdorff (and also packing) dimension 0. Thus, for example, there
is at most a zero-dimensional set of parameters $1/2<\lambda<1$ such
that the corresponding Bernoulli convolution has dimension $<1$,
and similarly for Sinai's problem on iterated function systems that
contract on average.

A central ingredient of the proof is an inverse theorem for the growth
of entropy of convolutions of probability measures $\mu,\nu$ on $\mathbb{R}$.
For the dyadic partition $\mathcal{D}_{n}$ of $\mathbb{R}$ into
cells of side $2^{-n}$, we show that if $\frac{1}{n}H(\nu*\mu,\mathcal{D}_{n})\leq\frac{1}{n}H(\mu,\mathcal{D}_{n})+\delta$,
then for $1\leq i\leq n$, either the restriction of $\mu$ to most
elements of $\mathcal{D}_{i}$ are close to uniform, or the restriction
of $\nu$ to most elements of $\mathcal{D}_{i}$ are close to atomic.
This should be compared to results in additive combinatorics that
give the global structure of measures $\mu$ satisfying $\frac{1}{n}H(\mu*\mu,\mathcal{D}_{n})\leq\frac{1}{n}H(\mu,\mathcal{D}_{n})+O(\frac{1}{n})$. 
\end{abstract}

\section{\label{sec:Introduction}Introduction}

The simplest examples of fractal sets and measures are self-similar
sets and measures on the line. These are objects that, like the classical
middle-third Cantor set, are made up of finitely many scaled copies
of themselves. When these scaled copies are sufficiently separated
from each other the small-scale structure is relatively easy to understand,
and, in particular, there is a closed formula for the dimension. If
one does not assume this separation, however, the picture becomes
significantly more complicated, and it is a longstanding open problem
to compute the dimension. This problem has spawned a number of related
conjectures, the most general of which is that, unless some of the
small-scale copies exactly coincide, the dimension should be equal
to the combinatorial upper bound, that is, the dimension one would
get if the small-scale copies did not intersect at all. Special cases
of this conjecture have received wide attention, e.g. Furstenberg's
projection problem and the Bernoulli convolutions problem. The purpose
of this paper is to shed some new light on these matters.

\subsection{\label{sub:Self-similar-sets}Self-similar sets and measures and
their dimension}

In this paper an iterated function system (IFS) will mean a finite
family $\Phi=\{\varphi_{i}\}_{i\in\Lambda}$ of linear contractions
of $\mathbb{R}$, written $\varphi_{i}(x)=r_{i}x+a_{i}$ with $|r_{i}|<1$
and $a_{i}\in\mathbb{R}$. To avoid trivialities we assume throughout
that there are at least two distinct contractions. A self similar
set is the attractor of such a system, i.e. the unique compact set
$\emptyset\neq X\subseteq\mathbb{R}$ satisfying 
\begin{equation}
X=\bigcup_{i\in\Lambda}\varphi_{i}X.\label{eq:self-similar-set}
\end{equation}
The self-similar measure associated to $\Phi$ and a probability vector
$(p_{i})_{i\in\Lambda}$ is the unique Borel probability measure $\mu$
on $\mathbb{R}^{d}$ satisfying
\begin{equation}
\mu=\sum_{i\in\Lambda}p_{i}\cdot\varphi_{i}\mu.\label{eq:self-similar-measure}
\end{equation}
Here $\varphi\mu=\mu\circ\varphi^{-1}$ denotes the push-forward of
$\mu$ by $\varphi$.

When the images $\varphi_{i}X$ are disjoint or satisfy various weaker
separation assumptions, the small-scale structure of self-similar
sets and measures is quite well understood. In particular the Hausdorff
dimension $\dim X$ of $X$ is equal to the similarity dimension%
\footnote{This notation is imprecise, since the similarity dimension depends
on the IFS $\Phi$ rather than the attractor $X$, but the meaning
should always be clear from the context. A similar remark holds for
the similarity dimension of measures.%
} $\sdim X$, i.e. the unique solution $s\geq0$ of the equation $\sum|r_{i}|^{s}=1$.
With the dimension of a measure $\theta$ defined by%
\footnote{This is the lower Hausdorff dimension. There are many other notions
of dimension but for self-similar measures all the major ones coincide
since such measures are exact dimensional \cite{FengHu09}. %
} 
\[
\dim\theta=\inf\{\dim E\,:\,\theta(E)>0\},
\]
and assuming again sufficient separation of the images $\varphi_{i}X$,
the dimension $\dim\mu$ of a self-similar measure $\mu$ is equal
to the similarity dimension of $\mu$, defined by 
\[
\sdim\mu=\frac{\sum p_{i}\log p_{i}}{\sum p_{i}\log r_{i}}.
\]

It is when the images $\varphi_{i}X$ have significant overlap that
computing the dimension becomes difficult, and much less is known.
One can give trivial bounds: the dimension is never greater than the
similarity dimension, and it is never greater than the dimension of
the ambient space $\mathbb{R}$, which is 1. Hence 
\begin{eqnarray}
\dim X & \leq & \min\{1,\sdim X\}\label{eq:similarity-dimension-bound}\\
\dim\mu & \leq & \min\{1,\sdim\mu\}.\label{eq:similarity-bound-for-measures}
\end{eqnarray}
However, without special combinatorial assumptions on the IFS, current
methods are unable even to decide whether or not equality holds in
\eqref{eq:similarity-dimension-bound} and \eqref{eq:similarity-bound-for-measures},
let alone compute the dimension exactly. The exception is when there
are sufficiently many exact overlaps among the ``cylinders'' of
the IFS. More precisely, for $i=i_{1}\ldots i_{n}\in\Lambda^{n}$
write
\[
\varphi_{i}=\varphi_{i_{1}}\circ\ldots\circ\varphi_{i_{n}}.
\]
One says that exact overlaps occur if there is an $n$ and distinct
$i,j\in\Lambda^{n}$ such that $\varphi_{i}=\varphi_{j}$ (in particular
the images $\varphi_{i}X$ and $\varphi_{j}X$ coincide).%
\footnote{If $i\in\Lambda^{k}$, $j\in\Lambda^{m}$ and $\varphi_{i}=\varphi_{j}$,
then $i$ cannot be a proper prefix of $j$ and vice versa, so $ij,ji\in\Lambda^{k+m}$
are distinct and $\varphi_{ij}=\varphi_{ji}$. Thus exact overlaps
occurs also if there is exact coincidence of cylinders at ``different
generations''. Stated differently, exact overlaps means that the
semigroup generated by the $\varphi_{i}$, $i\in\Lambda$, is not
freely generated by them.%
} If this occurs then $X$ and $\mu$ can be expressed using an IFS
$\Psi$ which is a proper subset of $\{\varphi_{i}\}_{i\in\Lambda^{n}}$,
and a strict inequality in \eqref{eq:similarity-dimension-bound}
and \eqref{eq:similarity-bound-for-measures} sometimes follows from
the corresponding bound for $\Psi$.

\subsection{\label{sub:Main-results}Main results}

The present work was motivated by the folklore conjecture that \emph{the
occurrence of exact overlaps is the only mechanism which can lead
to a strict inequality in} \eqref{eq:similarity-dimension-bound}
\emph{and }\eqref{eq:similarity-bound-for-measures} (see e.g. \cite[question 2.6]{PeresSolomyak2000b}).
Our main result lends some support to the conjecture and proves some
special cases of it. All of our results hold, with suitable modifications,
in higher dimensions, but this will appear separately.

Fix $\Phi=\{\varphi_{i}\}_{i\in\Lambda}$ as in the previous section
and for $i\in\Lambda^{n}$ write $r_{i}=r_{i_{1}}\cdot\ldots\cdot r_{i_{n}}$,
which is the contraction ratio of $\varphi_{i}$. Define the distance
between the cylinders associated to $i,j\in\Lambda^{n}$ by
\[
d(i,j)=\left\{ \begin{array}{cc}
\infty & r_{i}\neq r_{j}\\
|\varphi_{i}(0)-\varphi_{j}(0)| & r_{i}=r_{j}
\end{array}\right..
\]
Note that $d(i,j)=0$ if and only if $\varphi_{i}=\varphi_{j}$ and
that the definition is unchanged if $0$ is replaced by any other
point. For $n\in\mathbb{N}$ let 
\[
\Delta_{n}=\min\{d(i,j)\,:\, i,j\in\Lambda^{n}\,,\, i\neq j\}.
\]
Let us make a few observations:
\begin{itemize}
\item [-] Exact overlaps occur if and only if $\Delta_{n}=0$ for some
$n$ (equivalently all sufficiently large $n$). 
\item [-] $\Delta_{n}\rightarrow0$ exponentially. Indeed, the points $\varphi_{i}(0)$,
$i\in\Lambda^{n}$, can be shown to lie in a bounded interval independent
of $n$, and the exponentially many sequences $i\in\Lambda^{n}$ give
rise to only polynomially many contraction ratios $r_{i}$. Therefore
there are distinct $i,j\in\Lambda^{n}$ with $r_{i}=r_{j}$ and $|\varphi_{i}(0)-\varphi_{j}(0)|<|\Lambda|^{-(1-o(1))n}$.
\item [-]There can also be an exponential lower bound for $\Delta_{n}$.
This occurs when the images $\varphi_{i}(X)$, $i\in\Lambda$, are
disjoint, or under the open set condition, but also sometimes without
separation as in Garsia's example \cite{Garsia1962} or the cases
discussed in Theorems \ref{cor:algebraic-parameters} and \ref{thm:furstenberg}
below. 
\end{itemize}
Our main result on self-similar measures is the following.
\begin{thm}
\label{thm:main-individual}If $\mu$ is a self-similar measure on
$\mathbb{R}$ and if $\dim\mu<\min\{1,\sdim\mu\}$, then $\Delta_{n}\rightarrow0$
super-exponentially, i.e. $\lim(-\frac{1}{n}\log\Delta_{n})=\infty$. 
\end{thm}
The conclusion is about $\Delta_{n}$, which is determined by the
IFS $\Phi$, not by the measure. Thus, if the conclusion fails, then
$\dim\mu=\sdim\mu$ for every self-similar measure of $\Phi$. 
\begin{cor}
\label{cor:main-self-similar-sets}If $X$ is the attractor of an
IFS on $\mathbb{R}$ and if $\dim X<\min\{1,\sdim X\}$, then $\lim(-\frac{1}{n}\log\Delta_{n})=\infty$. \end{cor}
\begin{proof}
The self-similar measure $\mu$ associated to the probabilities $p_{i}=r_{i}^{\sdim X}$
satisfies $\sdim\mu=\sdim X$. Since $\mu(X)=1$ we have $\dim\mu\leq\dim X$,
so by hypothesis $\dim\mu<\min\{1,\sdim\mu\}$, and by the theorem,
$\Delta_{n}\rightarrow0$ super-exponentially. 
\end{proof}
Theorem \ref{thm:main-individual} is derived from a more quantitative
result about the entropy of finite approximations of $\mu$. Write
$H(\mu,\mathcal{E})$ for the Shannon entropy of a measure $\mu$
with respect to a partition $\mathcal{E}$, and $H(\mu,\mathcal{E}|\mathcal{F})$
for the conditional entropy on $\mathcal{F}$; see Section \ref{sub:Preliminaries-on-entropy}.
For $n\in\mathbb{Z}$ the dyadic partitions of $\mathbb{R}$ into
intervals of length $2^{-n}$ is 
\[
\mathcal{D}_{n}=\{[\frac{k}{2^{n}},\frac{k+1}{2^{n}})\,:\, k\in\mathbb{Z}\}.
\]
For $t\in\mathbb{R}$ we also write $\mathcal{D}_{t}=\mathcal{D}_{[t]}$.
We remark that $\liminf\frac{1}{n}H(\theta,\mathcal{D}_{n})\geq\dim\theta$
for any probability measure $\theta$, and the limit exists and is
equal to $\dim\theta$ when $\theta$ is exact dimensional, which
is the case for self-similar measures \cite{FengHu09}.

We first consider the case that $\Phi$ is uniformly contracting,
i.e. that all $r_{i}$ are equal to some fixed $r$. Fix a self-similar
measure $\mu$ defined by a probability vector $(p_{i})_{i\in\Lambda}$
and for $i\in\Lambda^{n}$ write $p_{i}=p_{i_{1}}\cdot\ldots\cdot p_{i_{n}}$.
Without loss of generality one can assume that $0$ belongs to the
attractor $X$. Define the $n$-th generation approximation of $\mu$
by 
\begin{equation}
\nu^{(n)}=\sum_{i\in\Lambda^{n}}p_{i}\cdot\delta_{\varphi_{i}(0)}.\label{eq:generation-n-measure}
\end{equation}
This is a probability measure on $X$ and $\nu^{(n)}\rightarrow\mu$
weakly. Moreover, writing 
\[
n'=n\log_{2}(1/r),
\]
$\nu^{(n)}$ closely resembles $\mu$ up to scale $2^{-n'}=r^{n}$
in the sense that 
\[
\lim_{n\rightarrow\infty}\frac{1}{n'}H(\nu^{(n)},\mathcal{D}_{n'})=\dim\mu.
\]
The main question we are interested in is the behavior of $\nu^{(n)}$
at smaller scales. Observe that the entropy $H(\nu^{(n)},\mathcal{D}_{n'})$
of $\nu^{(n)}$ at scale $2^{-n'}$ may not exhaust the entropy $H(\nu^{(n)})$
of $\nu^{(n)}$ as a discrete measure (i.e. with respect to the partition
into points). If there is substantial excess entropy it is natural
to ask at what scale and at what rate it appears; it must appear eventually
because $\lim_{k\rightarrow\infty}H(\nu^{(n)},\mathcal{D}_{k})=H(\nu^{(n)})$.
The excess entropy at scale $k$ relative to the entropy at scale
$n'$ is just the conditional entropy $H(\nu^{(n)},\mathcal{D}_{k}|\mathcal{D}_{n'})=H(\nu^{(n)},\mathcal{D}_{k})-H(\nu^{(n)},\mathcal{D}_{n'})$. 
\begin{thm}
\label{thm:main-individual-entropy-1}Let $\mu$ be a self-similar
measure on $\mathbb{R}$ defined by an IFS with uniform contraction
ratios. Let $\nu^{(n)}$ be as above. If $\dim\mu<1$, then 
\begin{equation}
\lim_{n\rightarrow\infty}\frac{1}{n'}H(\nu^{(n)},\mathcal{D}_{qn'}|\mathcal{D}_{n'})=0\quad\mbox{ for every }q>1.\label{eq:36}
\end{equation}

\end{thm}
Note that we assume $\dim\mu<1$ but not necessarily $\dim\mu<\sdim\mu$.
The statement is valid when $\dim\mu=\sdim\mu<1$, although for rather
trivial reasons.

We now formulate the result in the non-uniformly contracting case.
Let 
\[
r=\prod_{i\in\Lambda}r_{i}^{p_{i}}
\]
so that $\log r$ is the average logarithmic contraction ratio when
$\varphi_{i}$ is chosen randomly with probability $p_{i}$. Note
that, by the law of large numbers, with probability tending to $1$,
an element $i\in\Lambda^{n}$ chosen according to the probabilities
$p_{i}$ will satisfy $r_{i}=r^{n(1+o(1))}=2^{n'(1+o(1))}$. 

With this definition and $\nu^{(n)}$ defined as before, the theorem
above holds as stated, but note that now the partitions $\mathcal{D}_{k}$
are not suitable for detecting exact overlaps, since $\varphi_{i}(0)=\varphi_{j}(0)$
may happen for some $i,j\in\Lambda^{n}$ with $r_{i}\neq r_{j}$.
To correct this define the probability measure $\widetilde{\nu}^{(n)}$
on $\mathbb{R}\times\mathbb{R}$ by 
\[
\widetilde{\nu}^{(n)}=\sum_{i\in\Lambda^{n}}\delta_{(\varphi_{i}(0),r_{i})}
\]
and the partition of $\mathbb{R}\times\mathbb{R}$ given by 
\[
\widetilde{\mathcal{D}}_{n}=\mathcal{D}_{n}\times\mathcal{F},
\]
where $\mathcal{F}$ is the partition of $\mathbb{R}$ into points. 
\begin{thm}
\label{thm:main-individual-entropy-1-1}Let\, $\mu$ be a self-similar
measure on $\mathbb{R}$ and $\widetilde{\nu}^{(n)}$ as above. If
$\dim\mu<1$, then 
\begin{equation}
\lim_{n\rightarrow\infty}\frac{1}{n'}H(\widetilde{\nu}^{(n)},\mathcal{\widetilde{D}}_{qn'}|\widetilde{\mathcal{D}}_{n'})=0\quad\mbox{ for every }q>1.\label{eq:36-1}
\end{equation}

\end{thm}
To derive Theorem \ref{thm:main-individual}, let $\mu$ be as in
the last theorem with $\dim\mu<\min\{1,\sdim\mu\}$. The conclusion
of the last theorem is equivalent to $\frac{1}{n'}H(\widetilde{\nu}^{(n)},\widetilde{\mathcal{D}}_{qn'})\rightarrow\dim\mu$
for every $q>1$. Hence for a given $q$ and all sufficiently large
$n$ we will have $\frac{1}{n'}H(\widetilde{\nu}^{(n)},\widetilde{\mathcal{D}}_{qn'})<\sdim\mu$.
Since $\widetilde{\nu}^{(n)}=\sum_{i\in\Lambda^{n}}p_{i}\cdot\delta_{(\varphi_{i}(0),r_{i})}$,
if each pair $(\varphi_{i}(0),r_{j})$ belonged to a different atom
of $\widetilde{\mathcal{D}}_{qn'}$ then we would have $\frac{1}{n'}H(\widetilde{\nu}^{(n)},\widetilde{\mathcal{D}}_{qn'})=-\frac{1}{n\log(1/r)}\sum_{i\in\Lambda^{n}}p_{i}\log p_{i}=\sdim\mu$,
a contradiction. Thus there must be distinct $i,j\in\Lambda^{n}$
for which $(\varphi_{i}(0),r_{i})$, $(\varphi_{j}(0),r_{j})$ lie
in the same atom of $\widetilde{\mathcal{D}}_{qn'}$, giving $\Delta_{n}<2^{-qn'}$.

\subsection{\label{sub:Outline-of-the-proof}Outline of the proof}

Let us say a few words about the proofs. For simplicity we discuss
Theorem \ref{thm:main-individual-entropy-1}, where there is a common
contraction ratio $r$ to all the maps. For a self similar measure
$\mu=\sum_{i\in\Lambda}p_{i}\cdot\varphi_{i}\mu$, iterate this relation
$n$ times to get $\mu=\sum_{i\in\Lambda^{n}}p_{i}\cdot\varphi_{i}\mu$.
Since each $\varphi_{i}$, $i\in\Lambda^{n}$, contracts by $r^{n}$,
all the measures $\varphi_{i}\mu$, $i\in\Lambda^{n}$, are translates
of each other, the last identity can be re-written as a convolution
\[
\mu=\nu^{(n)}*\tau^{(n)},
\]
where as before $\nu^{(n)}=\sum_{i\in\Lambda^{n}}p_{i}\cdot\delta_{\varphi_{i}(0)}$,
and $\tau^{(n)}$ is $\mu$ scaled down by $r^{n}$. 

Fix $q$ and write $a\approx b$ to indicate that the difference tends
to $0$ as $n\rightarrow\infty$. From the entropy identity $H(\mu,\mathcal{D}_{(q+1)n'})=H(\mu,\mathcal{D}_{n'})+H(\mu,\mathcal{D}_{(q+1)n'}|\mathcal{D}_{n'})$
and the fact that $\frac{1}{n'}H(\mu,\mathcal{D}_{n'})\approx\frac{1}{n'}H(\nu^{(n)},\mathcal{D}_{n'})$,
we find that the mean entropy 
\[
A=\frac{1}{(q+1)n'}H(\mu,\mathcal{D}_{(q+1)n'})
\]
is approximately a convex combination $A\approx\frac{1}{(q+1)}B+\frac{q}{(q+1)}C$
of the mean entropy 
\[
B=\frac{1}{n'}H(\nu^{(n)},\mathcal{D}_{n'}),
\]
and the mean conditional entropy 
\[
C=\frac{1}{qn'}H(\mu,\mathcal{D}_{(q+1)n'}|\mathcal{D}_{n'})=\sum_{I\in\mathcal{D}_{n'}}\mu(I)\cdot\frac{1}{qn'}H(\nu_{I}^{(n)}*\tau^{(n)},\mathcal{D}_{(q+1)n'}),
\]
where $\nu_{I}^{(n)}$ is the normalized restriction of $\nu^{(n)}$
on $I$. Since $A\approx\dim\mu$ and $B\approx\dim\mu$, we find
that $C\approx\dim\mu$ as well. On the other hand we also have $\frac{1}{qn'}h(\tau^{(n)},\mathcal{D}_{(q+1)n'})\approx\dim\mu$.
Thus by the expression above, $C$ is an average of terms each of
which is close to the mean, and therefore most of them are equal to
the mean. We find that 
\begin{equation}
\frac{1}{qn'}H(\nu_{I}^{(n)}*\tau^{(n)},\mathcal{D}_{(q+1)n'})\approx C\approx\dim\mu\approx\frac{1}{qn'}H(\tau^{(n)},\mathcal{D}_{(q+1)n'})\label{eq:approximate-non-growth}
\end{equation}
for large $n$ and ``typical'' $I\in\mathcal{D}_{n'}$. The argument
is then concluded by showing that \eqref{eq:approximate-non-growth}
implies that either $\frac{1}{qn'}H(\tau^{(n)},\mathcal{D}_{(q+1)n'})\approx1$
(leading to $\dim\mu=1$), or that typical intervals $I$ satisfy
$\frac{1}{qn'}H(\nu_{I}^{(n)},\mathcal{D}_{(q+1)n'})\approx0$ (leading
to \eqref{eq:36}). 

Now, for a general pair of measures $\nu,\tau$ the relation $\frac{1}{n}H(\nu*\tau,\mathcal{D}_{n})\approx\frac{1}{n}H(\nu,\mathcal{D}_{n})$
analogous to \eqref{eq:approximate-non-growth}  does not have such
an implication. But, while we know nothing about the structure of
$\nu_{I}^{(n)}$, we do know that $\tau^{(n)}$, being self-similar,
is highly uniform at different scales. We will be able to utilize
this fact to draw the desired conclusion. Evidently, the main ingredient
in the argument is an analysis of the growth of measures under convolution,
which will occupy us starting in Section \ref{sec:Additive-combinatorics}.

\subsection{\label{sub:main-applications}Applications}

Theorem \ref{thm:main-individual} and its corollaries settle a number
of cases of the aforementioned conjecture. Specifically, in any class
of IFSs where one can prove that cylinders are either equal or exponentially
separated, the only possible cause of dimension drop is the occurrence
of exact overlaps. Thus,
\begin{thm}
\label{cor:algebraic-parameters}For IFSs on $\mathbb{R}$ defined
by algebraic parameters, there is a dichotomy: Either there are exact
overlaps or the attractor $X$ satisfies $\dim X=\min\{1,\sdim X\}$.\end{thm}
\begin{proof}
Let $\varphi_{i}(x)=r_{i}x+a_{i}$ and suppose $r_{i},a_{i}$ are
algebraic. For distinct $i,j\in\Lambda^{n}$ the distance $|\varphi_{i}(0)-\varphi_{j}(0)|$
is a polynomial of degree $n$ in $r_{i},a_{i}$, and hence is either
equal to $0$, or is $\geq s^{n}$ for some constant $s>0$ depending
only on the numbers $r_{i},a_{i}$ (see Lemma \ref{lem:Liouville-bound}).
Thus $\Delta_{n}\geq s^{n}$ and the conclusion follows from Corollary
\ref{cor:main-self-similar-sets}.
\end{proof}
There are a handful of cases where a similar argument can handle non-algebraic
parameters. Among these is a conjecture by Furstenberg from the 1970s,
asserting that if the ``one dimensional Sierpinski gasket'' 
\[
F=\left\{ \sum(i_{n},j_{n})3^{-n}\,:\,(i_{n},j_{n})\in\{(0,0),(1,0),(0,1)\}\right\} 
\]
is projected orthogonally to a line of irrational slope, then the
dimension of the image is $1$ (see e.g. \cite[question 2.5]{PeresSolomyak2000b}).%
\footnote{This was motivated by a dual conjecture asserting that any line $\ell$
of irrational slope meets $F$ in a of zero dimensional set, and this,
in turn, is an analog of similar conjectures arising in metric number
theory and layed out in \cite{Furstenberg70}. The intersections and
projections conjectures are related by the heuristic that for a map
$F\rightarrow\mathbb{R}$, a large image corresponds to small fibers,
but there is only an implication in one direction (the statement about
intersections implies the one about projections using \cite{Furstenberg08}).%
} It is more convenient to replace orthogonal projections with the
parametrized v linear maps $\pi_{t}:\mathbb{R}^{2}\rightarrow\mathbb{R}$
given by 
\[
\pi_{t}(x,y)=tx+y
\]
(up to a linear change of coordinates in the range, this represents
the orthogonal projection to the line with slope $-1/t$). One may
verify that the image $F_{t}=\pi_{t}F$ is the self-similar defined
by the contractions 
\begin{equation}
x\mapsto\frac{1}{3}x\quad,\quad x\mapsto\frac{1}{3}(x+1)\quad,\quad x\mapsto\frac{1}{3}(x+t).\label{eq:sierpinski-gasket-contractions}
\end{equation}
Therefore $\sdim F_{t}=1$ for all $t$, and it is not hard to show
that exact overlaps occur only for certain rational values of $t$.
Thus, Furstenberg's conjecture is a special case of the motivating
conjecture of this paper.

From general considerations such as Marstrand's theorem, we know that
$\dim F_{t}=1$ for a.e. $t$, and Kenyon showed that this holds also
for a dense $G_{\delta}$ set of $t$ \cite{Kenyon97}. In the same
paper Kenyon also classified those rational $t$ for which $\dim F_{t}=1$,
and showed that $F_{t}$ has Lebesgue measure $0$ for all irrational
$t$ (strengthening the conclusion of a general theorem of Besicovitch
that gives this for a.e. $t$). For some other partial results see
\cite{SwiatekVeerman2002}. 
\begin{thm}
\label{thm:furstenberg} If $t\notin\mathbb{Q}$ then $\dim F_{t}=1$.\end{thm}
\begin{proof}
Fix $t$, and suppose that $\dim F_{t}<1$. Let $\Lambda=\{0,1,t\}$
and $\varphi_{i}(x)=\frac{1}{3}(x+i)$, so $F_{t}$ is the attractor
of $\{\varphi_{i}\}_{i\in\Lambda}$. For $i\in\Lambda^{n}$ one may
check that $\varphi_{i}(0)=\sum_{k=1}^{n}i_{k}3^{-k}$. Inserting
this into the difference $\varphi_{i}(0)-\varphi_{j}(0)$ we can separate
the terms that are multiplied by $t$ from those that are not, and
we find that $|\varphi_{i}(0)-\varphi_{j}(0)|=p_{i,j}-t\cdot q_{i,j}$
for rational numbers $p_{i,j},q_{i,j}$ belonging to the set 
\[
X_{n}=\{\sum_{i=1}^{n}a_{i}3^{-i}\,:\, a_{i}\in\{\pm1,0\}\}
\]
Therefore there are $p_{n},q_{n}\in X_{n}$ such that $\Delta_{n}=|p_{n}-tq_{n}|$,
so by Corollary \ref{cor:main-self-similar-sets}, 
\begin{equation}
|p_{n}-t\cdot q_{n}|<30^{-n}\qquad\mbox{for large enough }n\label{eq:75}
\end{equation}

If $q_{n}=0$ for $n$ satisfying \eqref{eq:75} then $|p_{n}|<30^{-n}$,
but, since $p_{n}$ is rational with denominator $3^{n}$, this can
only happen if $p_{n}=0$. This in turn implies that $\Delta_{n}=0$,
i.e. there are exact overlaps, so $t\in\mathbb{Q}$. 

On the other hand suppose $q_{n}\neq0$ for all large $n$. Since
$q_{n}$ is a non-zero rational with denominator $3^{n}$ we have
$q_{n}\geq3^{-n}$. Dividing \eqref{eq:75} by $q_{n}$ we get $|t-p_{n}/q_{n}|<10^{-n}$.
Subtracting successive terms, by the triangle inequality we have 
\[
|\frac{p_{n+1}}{q_{n+1}}-\frac{p_{n}}{q_{n}}|<2\cdot10^{-n}\qquad\mbox{for large enough }n.
\]
But $p_{n},q_{n},p_{n+1},q_{n+1}\in X_{n+1}$, so $p_{n+1}/q_{n+1}-p_{n}/q_{n}$
is rational with denominator $\leq9^{n+1}$, giving 
\[
|\frac{p_{n+1}}{q_{n+1}}-\frac{p_{n}}{q_{n}}|\neq0\quad\implies\quad|\frac{p_{n+1}}{q_{n+1}}-\frac{p_{n}}{q_{n}}|\geq9^{-(n+1)}.
\]
Since $9^{-(n+1)}\leq2\cdot10^{-n}$ is impossible for large $n$,
the last two equations imply that $p_{n}/q_{n}=p_{n+1}/q_{n+1}$ for
all large $n$. Therefore there is an $n_{0}$ such that $|t-p_{n_{0}}/q_{n_{0}}|<10^{-n}$
for $n>n_{0}$ which gives $t=p_{0}/q_{0}$.
\end{proof}
The argument above is due to B. Solomyak and P. Shmerkin and we thank
them for permission to include it here. Similar considerations work
in a few other cases, but one already runs into difficulties if in
the example above we replace the contraction ratio 1/3 with any non-algebraic
$0<r<1$ (see also the discussion following Theorem \ref{thm:BC}
below). 

In the absence of a resolution of the general conjecture, we turn
to parametric families of self-similar sets and measures. The study
of parametric families of general sets and measures is classical;
examples include the projection theorems of Besicovitch and Marstrand
and more recent results like those of Peres-Schlag \cite{PeresSchlag2000}
and Bourgain \cite{Bourgain2010}. When the sets and measures in question
are self-similar we shall see that the general results can be strengthened
considerably. 

Let $I$ be a set of parameters, let $r_{i}:I\rightarrow(-1,1)\setminus\{0\}$
and $a_{i}:I\rightarrow\mathbb{R}$, $i\in\Lambda$. For each $t\in I$
define $\varphi_{i,t}:\mathbb{R}\rightarrow\mathbb{R}$ by $\varphi_{i,t}(x)=r_{i}(t)(x-a_{i}(t))$.
For a sequence $i\in\Lambda^{n}$ let $\varphi_{i,t}=\varphi_{i_{1},t}\circ\ldots\circ\varphi_{i_{n},t}$
and define 
\begin{eqnarray}
\Delta_{i,j}(t) & = & \varphi_{i,t}(0)-\varphi_{j,t}(0).\label{eq:44}
\end{eqnarray}
The quantity $\Delta_{n}=\Delta_{n}(t)$ associated as in the previous
section to the IFS $\{\varphi_{i,t}\}_{i\in\Lambda}$ is not smaller
than the minimum of $|\Delta_{i,j}(t)|$ over distinct $i,j\in\Lambda^{n}$
(since it is the minimum over pairs $i,j$ with $r_{i}=r_{j}$). Thus,
$\Delta_{n}\rightarrow0$ super-exponentially implies that $\min\{|\Delta_{i,j}(t)|\,,\, i,j\in\Lambda^{n}\}\rightarrow0$
super-exponentially as well, so Theorem \ref{thm:main-individual}
has the following formal implication.
\begin{thm}
\label{thm:description-of-exceptional-params}Let $\Phi_{t}=\{\varphi_{i,t}\}$
be a parametrized IFS as above. For every $\varepsilon>0$ let\textup{\emph{
\begin{equation}
E_{\varepsilon}=\bigcup_{N=1}^{\infty}\,\bigcap_{n>N}\left(\bigcup_{i,j\in\Lambda^{n}}(\Delta_{i,j})^{-1}(-\varepsilon^{n},\varepsilon^{n})\right)\label{eq:23}
\end{equation}
}}and 
\begin{equation}
E=\bigcap_{\varepsilon>0}E_{\varepsilon}.\label{eq:52}
\end{equation}
Then for $t\in I\setminus E$, for every probability vector $p=(p_{i})$
the associated self-similar measure $\mu_{t}$ of $\Phi_{t}$ satisfies
$\dim\mu_{t}=\min\{1,\sdim\mu_{t}\}$, and the attractor $X_{t}$
of $\Phi_{t}$ satisfies $\dim X_{t}=\min\{1,\sdim X_{t}\}$.
\end{thm}
Our goal is to show that the set $E$ defined in the theorem above
is small. We restrict ourselves to the case that $I\subseteq\mathbb{R}$
is a compact interval; a multi-parameter version will appear in \cite{Hochman2012b}.
Extend the definition of $\Delta_{i,j}$ to infinite sequences $i,j\in\Lambda^{\mathbb{N}}$
by 
\begin{eqnarray}
\Delta_{i,j}(t) & = & \lim_{n\rightarrow\infty}\Delta_{i_{1}\ldots i_{n},j_{1}\ldots j_{n}}(t).\label{eq:50}
\end{eqnarray}
Convergence is uniform over $I$ and $i,j$, and if $a_{i}(\cdot)$
and $r_{i}(\cdot)$ are real analytic in a neighborhood of $I$ then
so are the functions $\Delta_{i,j}(\cdot)$. 
\begin{thm}
\label{thm:main-parametric}Let $I\subseteq\mathbb{R}$ be a compact
interval, let $r:I\rightarrow(-1,1)\setminus\{0\}$ and $a_{i}:I\rightarrow\mathbb{R}$
be real analytic, and let $\Phi_{t}=\{\varphi_{i,t}\}_{i\in\Lambda}$
be the associated parametric family of IFSs, as above. Suppose that
\[
\forall i,j\in\Lambda^{\mathbb{N}}\quad\left(\;\Delta_{i,j}\equiv0\mbox{ on }I\quad\iff\quad i=j\;\right).
\]
Then \textup{\emph{the set $E$ of ``exceptional'' parameters in
Theorem \ref{thm:description-of-exceptional-params} has}} Hausdorff
and packing dimension $0$. 
\end{thm}
The condition in the theorem is extremely mild. Essentially it means
that the family does not have overlaps ``built in''. For an example
where the hypothesis fails, consider the case that there are $i\neq j$
with $\varphi_{i,t}=\varphi_{j,t}$ for all $t$. In this case the
conclusion sometimes fails as well.

Most existing results on parametric families of IFSs are based on
the so-called transversality method, introduced by Pollicott and Simon
\cite{PollicottSimon1995} and developed, among others, by Solomyak
\cite{Solomyak1995} and Peres-Schlag \cite{PeresSchlag2000}. Theorem
\ref{thm:main-parametric} is based on a similar but much weaker ``higher
order'' transversality condition, which is automatically satisfied
under the stated hypothesis. We give the details in Section \ref{sub:Transversality-and-exceptions}.
See \cite{ShmerkinSolomyak2006} for an effective derivation of higher-order
transversality in certain contexts.

As a demonstration we apply this to the Bernoulli convolutions problem.
For $0<\lambda<1$ let $\nu_{\lambda}$ denote the distribution of
the real random variable $\sum_{n=0}^{\infty}\pm\lambda^{n}$, where
the signs are chosen i.i.d. with equal probabilities. The name derives
from the fact that $\nu_{\lambda}$ is the infinite convolution of
the measures $\frac{1}{2}\left(\delta_{-\lambda^{n}}+\delta_{\lambda^{n}}\right)$,
$n=0,1,2,\ldots$, but the pertinent fact for us is that $\nu_{\lambda}$
is a self-similar measure, given by assigning equal probabilities
to the contractions
\begin{equation}
\varphi_{\pm}(x)=\lambda x\pm1.\label{eq:BC-contractions}
\end{equation}

For $\lambda<\frac{1}{2}$ the measure is supported on a self-similar
Cantor set of dimension $<1$, but for $\lambda\in[\frac{1}{2},1)$
the support is an interval, and it is a well-known open problem to
determine whether it is absolute continuous. Exact overlaps can occur
only for certain algebraic $\lambda$, and Erd\H{o}s showed that
when $\lambda^{-1}$ is a Pisot number $\nu_{\lambda}$ is in fact
singular \cite{Erdos1939}. No other parameters $\lambda\in[\frac{1}{2},1)$
are known for which $\nu_{\lambda}$ is singular. In the positive
direction, it is known that $\nu_{\lambda}$ is absolutely continuous
for a.e. $\lambda\in[1/2,1)$ (Solomyak \cite{Solomyak1995}) and
the set of exceptional $\lambda\in[a,1)$ has dimension $<1-C(a-1/2)$
for some $C>0$ (Peres-Schlag \cite{PeresSchlag2000}) and its dimension
tends to $0$ as $a\rightarrow1$ (Erd\H{o}s \cite{Erdos1940}).

We shall consider the question of when $\dim\nu_{\lambda}=1$. This
is weaker than absolute continuity but little more seems to be known
about this question except the relatively soft fact that the set of
parameters with $\dim\nu_{\lambda}=1$ is also topologically large
(contains a dense $G_{\delta}$ set); see \cite{PeresSchlagSolomyak00}.
In particular the only parameters $\lambda\in[1/2,1)$ for which $\dim\nu_{\lambda}<1$
is known are inverses of Pisot numbers (Alexander-Yorke \cite{AlexanderYorke1984}).
We also note that in many of the problems related to Bernoulli convolutions
it is the dimension of $\nu_{\lambda}$, rather than its absolute
continuity, that are relevant. For discussion of some applications
see \cite[Section 8]{PeresSchlagSolomyak00} and \cite{PrzytyckiUrbanski1989}.
\begin{thm}
\label{thm:BC}$\dim\nu_{\lambda}=1$ outside a set of $\lambda$
of dimension $0$. Furthermore, the exceptional parameters for which
$\dim\nu_{\lambda}<1$ are ``nearly algebraic'' in the sense that
for every $0<\theta<1$ and all large enough $n$, there is a polynomial
$p_{n}(t)$ of degree $n$ and coefficients $0,\pm1$, such that $|p_{n}(\lambda)|<\theta^{n}$.\end{thm}
\begin{proof}
Take the parametrization $r(t)=t$, $a_{\pm}(t)=\pm1$ for $t\in[1/2,1-\varepsilon]$.
Then $\Delta_{i,j}(t)=\sum(i_{n}-j_{n})\cdot t^{n}$ and this vanishes
identically if and only if $i=j$, confirming the hypothesis of Theorem
\ref{thm:main-parametric}. Since $\Delta_{n}(t)$ is a polynomial
of degree $n$ with coefficients $0,\pm1$, so the second statement
follows the description of the set $E$ in Theorem \ref{thm:main-parametric}.
\end{proof}
Arguing as in the proof of Theorem \ref{thm:furstenberg}, in order
to show that $\dim\nu_{\lambda}=1$ for all non-algebraic $\lambda$,
it would suffice to answer the following question in the affirmative:%
\footnote{In order to show that an ``almost-root'' of a polynomial is close
to an acrual root one can rely on the classical transversality arguments,
e.g. \cite{Solomyak1995}.%
}
\begin{question}
Let $\Pi_{n}$ denote the collection of polynomial of degree $\leq n$
with coefficients $0,\pm1$. Does there exist a constant $s>0$ such
that for $\alpha,\beta$ that are roots of polynomials in $\Pi_{n}$
either $\alpha=\beta$ or $|\alpha-\beta|>s^{n}$?
\end{question}
Classical bounds imply that this for $s\sim1/n$, but we have not
found an answer to the question in the literature.

Another problem to which our methods apply is the Keane-Smorodinsky
\{0,1,3\}-problem. For details about the problem we refer to Pollicott-Simon
\cite{PollicottSimon1995} or Keane-Smorodinsky-Solomyak \cite{KeaneSmorodinskySolomyak1995}. 

Finally, our methods also can be adapted with minor changes to IFSs
that ``contract on average'' \cite{NicolSidorovBroomhead2002}.
We restrict attention to a problem raised by Sinai \cite{PeresSimonSolomyak2006}
concerning the maps $\varphi_{-}:x\mapsto(1-\alpha)x-1$ and $\varphi_{+}:x\mapsto(1+\alpha)x+1$.
A composition of $n$ of these maps chosen i.i.d. with probability
$\frac{1}{2},\frac{1}{2}$ asymptotically contracts by approximately
$(1-\alpha^{2})^{n/2}$, and so for each $0<\alpha<1$ there is a
unique probability measure $\mu_{\alpha}$ on $\mathbb{R}$ satisfying
$\mu_{\alpha}=\frac{1}{2}\varphi_{-}\mu_{\alpha}+\frac{1}{2}\varphi_{+}\mu_{\alpha}$.
Little is known about the dimension or absolute continuity of $\mu_{\alpha}$
beyond upper bounds analogous to \eqref{eq:similarity-bound-for-measures}.
Some results in a randomized analog of this model have been obtained
by Peres, Simon and Solomyak \cite{PeresSimonSolomyak2006}. We prove
\begin{thm}
\label{thm:Sinais-problem} There is a set $E\subseteq(0,1)$ of Hausdorff
(and packing) dimension $0$ such that $\dim\mu_{\alpha}=\min\{1,\sdim\mu_{\alpha}\}$
for $\alpha\in(0,1)\setminus E$.
\end{thm}
For further discussion of this problem see Section \ref{sub:Applications}.

\subsection{Absolute continuity?}

There is a conjecture analogous to the one we began with, predicting
that if $\mu$ is a self-similar measure, $\sdim\mu>1$, and there
are no exact overlaps, then $\mu$ should be absolutely continuous
with respect to Lebesgue measure. The Bernoulli convolutions problem
discussed above is a special case of this conjecture.

Our methods at present are not able to tackle this problem. At a technical
level, whenever our methods give $\dim\mu=1$ it is a consequence
of showing that $H(\mu,\mathcal{D}_{n})=n-o(n)$. In contrast, absolute
continuity would require better asymptotics, e.g. $H(\mu,\mathcal{D}_{n})=n-O(1)$
(see \cite[Theorem 1.5]{Garsia1963}). More substantially, our arguments
do not distinguish between the critical $\sdim\mu=1$, where the conclusion
of the conjecture is generally false, and super-critical phase $\sdim\mu>1$,
so in their present form they cannot possibly give results about absolute
continuity.

The discussion above notwithstanding, shortly after this paper appeared
in preprint form, P. Shmerkin found an ingenious way to ``amplify''
our results on parametric families of self-similar measures and obtain
results about absolute continuity. For instance,
\begin{thm*}
[Shmerkin \cite{Shmerkin2013}] There is a set $E\subseteq(\frac{1}{2},1)$
of Hausdorff dimension $0$, such the Bernoulli convolution $\nu_{\lambda}$
is absolutely continuous for all $\lambda\in(\frac{1}{2},1)\setminus E$.
\end{thm*}
The idea of the proof is to split $\nu_{\lambda}$ as a convolution
$\nu'_{\lambda}*\nu''_{\lambda}$ of self-similar measures, with $\sdim\nu'_{\lambda}\geq1$
and $\sdim\nu''_{\lambda}>0$. By Theorem \ref{thm:main-parametric},
$\dim\nu'_{\lambda}=1$ outside a zero-dimensional set $E'$ of parameters.
On the other hand a classical argument of Erd\H{o}s and Kahane shows
that, outside a zero-dimensional set $E''$ of parameters, the Fourier
transform of $\nu''_{\lambda}$ has power decay. Taking $E=E'\cup E''$,
Shmerkin shows that $\nu_{\lambda}=\nu'_{\lambda}*\nu''_{\lambda}$
is absolutely continuous for $\lambda\in(\frac{1}{2},1)\setminus E$. 

At present the argument above is limited by the fact that $E''$ is
completely non-effective, so, unlike Theorem \ref{thm:main-individual},
it does not give a condition that applies to\emph{ individual }self-similar
measure, and does not provide concrete new examples of parameters
for which $\nu_{\lambda}$ is absolutely continuous. In contrast,
Corollary \ref{cor:algebraic-parameters} tells us that $\dim\nu_{\lambda}=1$
whenever $\lambda\in(\frac{1}{2},1)\cap\mathbb{Q}$, as well as other
algebraic examples. It remains a challenge to prove a similar result
for absolute continuity.

\subsection{\label{sub:Organization}Notation and organization of the paper }

The main ingredient in the proofs are our results on the growth of
convolutions of measures. We develop this subject in the next three
sections: Section \ref{sec:Additive-combinatorics} introduces the
statements and basic definitions, Section \ref{sec:Entropy-concentration-uniformity-saturation}
contains some preliminaries on entropy and convolutions, and Section
\ref{sec:Entropy-growth-for-convolutions} proves the main results
on convolutions. In Section \ref{sec:Parameterized-families-of-self-similar-measures}
we prove Theorem \ref{thm:main-individual} and the other main results.

We follow standard notational conventions. $\mathbb{N}=\{1,2,3,\ldots\}$.
All logarithms are to base $2$. $\mathcal{P}(X)$ is the space of
probability measures on $X$, endowed with the weak-{*} topology if
appropriate. We follow standard ``big $O$'' notation: $O_{\alpha}(f(n))$
is an unspecified function bounded in absolute value by $C_{\alpha}\cdot f(n)$
for some constant $C_{\alpha}$ depending on $\alpha$. Similarly
$o(1)$ is a quantity tending to $0$ as the relevant parameter $\rightarrow\infty$.
The statement ``for all $s$ and $t>t(s),\ldots$'' should be understood
as saying ``there exists a function $t(\cdot)$ such that for all
$s$ and $t>t(s),\ldots$''. If we want to refer to the function
$t(\cdot)$ outside the context where it is introduced we will designate
it as $t_{1}(\cdot)$, $t_{2}(\cdot)$, etc.

\subsection*{Acknowledgment}

I am grateful to Pablo Shmerkin and Boris Solomyak for many contributions
which have made this a better paper, and especially their permission
to include the derivation of Theorem \ref{thm:furstenberg}. I also
thank Nicolas de Saxce and Izabella Laba for their comments. This
project began during a visit to Microsoft Research in Redmond, Washington,
and I would like to thank Yuval Peres and the members of the theory
group for their hospitality.

\section{\label{sec:Additive-combinatorics}An inverse theorem for the entropy
of convolutions}

\subsection{\label{sub:Entropy-and-additive-combinatorics}Entropy and additive
combinatorics}

As we saw in Section \ref{sub:Outline-of-the-proof}, a key ingredient
in the proof of Theorems \ref{thm:main-individual-entropy-1} is an
analysis of the growth of measures under convolution. This subject
is of independent interest and will occupy us for a large part of
this paper. 

It will be convenient to introduce the normalized scale-$n$ entropy
\[
H_{n}(\mu)=\frac{1}{n}H(\mu,\mathcal{D}_{n}).
\]
Our aim is to obtain structural information about measures $\mu,\nu$
for which $\mu*\nu$ is small in the sense that 
\begin{equation}
H_{n}(\mu*\nu)\leq H_{n}(\mu)+\delta,\label{eq:mean-entropy-growth}
\end{equation}
where $\delta>0$ is small but fixed, and $n$ is large. 

This problem is a relative of classical ones in additive combinatorics
concerning the structure of sets $A,B$ whose sumset $A+B=\{a+b\,:\, a\in A\,,\, b\in B\}$
is appropriately small. The general principle is that when the sum
is small, the sets should have some algebraic structure. Results to
this effect are known as inverse theorems. For example the Freiman-Rusza
theorem asserts that if $|A+B|\leq C|A|$ then $A,B$ are close, in
a manner depending on $C$, to generalized arithmetic progressions%
\footnote{A generalized arithmetic progression is an affine image of a box in
a higher-dimensional lattice.%
} (the converse is immediate). For details and more discussion see
e.g \cite{TaoVu2006}. 

The entropy of a discrete measure corresponds to the logarithm of
the cardinality of a set, and convolution is the analog for measures
of the sumset operation. Thus the analog of the condition $|A+A|\leq C|A|$
is 
\begin{equation}
H_{n}(\mu*\mu)\leq H_{n}(\mu)+O(\frac{1}{n})\label{eq:O-of-1-entropy-growth}
\end{equation}
An entropy version of Freiman's theorem was recently proved by Tao
\cite{Tao2010}, who showed that if $\mu$ satisfies \eqref{eq:O-of-1-entropy-growth}
then it is close, in an appropriate sense, to a uniform measures on
a (generalized) arithmetic progression. 

The condition \eqref{eq:mean-entropy-growth}, however, significantly
weaker than \eqref{eq:O-of-1-entropy-growth} even when the latter
is specialized to $\nu=\mu$, and it is harder to draw conclusions
from it about the global structure of $\mu$. Consider the following
example. Start with an arithmetic progression of length $n_{1}$ and
gap $\varepsilon_{1}$, and put the uniform measure on it. Now split
each atom $x$ into an arithmetic progression of length $n_{2}$ and
gap $\varepsilon_{2}<\varepsilon_{1}/n_{2}$, starting at $x$ (so
the entire gap fits in the space between $x$ and the next atom).
Repeat this procedure $N$ times with parameters $n_{i},\varepsilon_{i}$,
and call the resulting measure $\mu$. Let $k$ be such that $\varepsilon_{N}$
is of order $2^{-k}$. It is not hard to verify that we can have $H_{k}(\mu)=1/2$
but $|H_{k}(\mu)-H_{k}(\mu*\mu)|$ arbitrarily small. This example
is actually the uniform measure on a (generalized) arithmetic progression,
as predicted by Freiman-type theorems, but the rank $N$ can be arbitrarily
large. Furthermore if one conditions $\mu$ on an exponentially small
subset of its support one gets another example with the similar properties
that is quite far from a generalized arithmetic progression. 

Our main contribution to this matter is Theorem \ref{thm:inverse-thm-Rd}
below, which shows that constructions like the one above are, in a
certain statistical sense, the only way that \eqref{eq:mean-entropy-growth}
can occur. We note that there is a substantial existing literature
on the growth condition $|A+B|\leq|A|^{1+\delta}$, which is the sumset
analog of \eqref{eq:mean-entropy-growth}. Such a condition appears
in the sum-product theorems of Bourgain-Katz-Tao \cite{BourgainKatzTao2004}
and in the work of Katz-Tao \cite{KatzTao2001}, and in the Euclidean
setting more explicitly in Bourgain's work on the Erd\H{o}s-Volkmann
conjecture \cite{Bourgain2003} and Marstrand-like projection theorems
\cite{Bourgain2010}. However we have not found a result in the literature
that meets our needs and, in any event, we believe that the formulation
given here will find further applications.

\subsection{\label{sub:Component-measures}Component measures }

The following notation will be needed in $\mathbb{R}^{d}$ as well
as $\mathbb{R}$. Let $\mathcal{D}_{n}^{d}=\mathcal{D}_{n}\times\ldots\times\mathcal{D}_{n}$
denote the dyadic partition of $\mathbb{R}^{d}$; we often suppress
the superscript when it is clear from the context. Let $\mathcal{D}_{n}(x)\in\mathcal{D}_{n}$
denote the unique level-$n$ dyadic cell containing $x$. For $D\in\mathcal{D}_{n}$
let $T_{D}:\mathbb{R}^{d}\rightarrow\mathbb{R}^{d}$ be the unique
homothety mapping $D$ to $[0,1)^{d}$. Recall that if $\mu\in\mathcal{P}(\mathbb{R})$
then $T_{D}\mu$ is the push-forward of $\mu$ through $T_{D}$ .
\begin{defn}
For $\mu\in\mathcal{P}(\mathbb{R}^{d})$ and a dyadic cell $D$ with
$\mu(D)>0$, the (raw) $D$-component of $\mu$ is
\[
\mu_{D}=\frac{1}{\mu(D)}\mu|_{D}
\]
and the (rescaled) $D$-component is 
\[
\mu^{D}=\frac{1}{\mu(D)}T_{D}(\mu|_{D}).
\]
For $x\in\mathbb{R}^{d}$ with $\mu(\mathcal{D}_{n}(x))>0$ we write
\begin{eqnarray*}
\mu_{x,n} & = & \mu_{\mathcal{D}_{n}(x)}\\
\mu^{x,n} & = & \mu^{\mathcal{D}_{n}(x)}.
\end{eqnarray*}
These measures, as $x$ ranges over all possible values for which
$\mu(\mathcal{D}_{n}(x))>0$, are called the level-$n$ components
of $\mu$.
\end{defn}
Our results on the multi-scale structure of $\mu\in\mathbb{R}^{d}$
are stated in terms of the behavior of random components of $\mu$,
defined as follows.%
\footnote{Definition \ref{def:component-distribution} is motivated by Furstenberg's
notion of a CP-distribution \cite{Furstenberg70,Furstenberg08,HochmanShmerkin2011},
which arise as limits as $N\rightarrow\infty$ of the distribution
of components of level $1,\ldots,N$. These limits have a useful dynamical
interpretation but in our finitary setting we do not require this
technology.%
}
\begin{defn}
\label{def:component-distribution}Let $\mu\in\mathcal{P}(\mathbb{R}^{d})$. 
\begin{enumerate}
\item A random level-$n$ component, raw or rescaled, is the random measure
$\mu_{D}$ or $\mu^{D}$, respectively, obtained by choosing $D\in\mathcal{D}_{n}$
with probability $\mu(D)$; equivalently, the random measure $\mu_{x,n}$
or $\mu^{x,n}$, respectively, with $x$ chosen according to $\mu$.
\item For a finite set $I\subseteq\mathbb{N}$, a random level-$I$ component,
raw or rescaled, is chosen by first choosing $n\in I$ uniformly,
and then (conditionally independently on the choice of $n$) choosing
a raw or rescaled level-$n$ component, respectively.
\end{enumerate}
\end{defn}
\begin{notation}
When the symbols $\mu^{x,i}$ and $\mu_{x,i}$ appear inside an expression
$\mathbb{P}\left(\ldots\right)$ or $\mathbb{E}\left(\ldots\right)$,
they will always denote random variables drawn according to the component
distributions defined above. The range of $i$ will be specified as
needed. 
\end{notation}
The definition is best understood with some examples. For $A\subseteq\mathcal{P}([0,1]^{d})$
we have
\begin{eqnarray*}
\mathbb{P}_{i=n}\left(\mu^{x,i}\in A\right) & = & \int1_{A}(\mu^{x,n})\, d\mu(x)\\
\mathbb{P}_{0\leq i\leq n}\left(\mu^{x,i}\in A\right) & = & \frac{1}{n+1}\sum_{i=0}^{n}\int1_{A}(\mu^{x,i})\, d\mu(x).
\end{eqnarray*}
This notation implicitly defines $x,i$ as random variables. Thus
if $A_{0},A_{1},\ldots\subseteq\mathcal{P}([0,1]^{d})$ and $D\subseteq[0,1]^{d}$
we could write 
\[
\mathbb{P}_{0\leq i\leq n}\left(\mu^{x,i}\in A_{i}\mbox{ and }x\in D\right)=\frac{1}{n+1}\sum_{i=0}^{n}\mu\left(x\,:\,\mu^{x,i}\in A_{i}\mbox{ and }x\in D\right).
\]
Similarly, for $f:\mathcal{P}([0,1)^{d})\rightarrow\mathbb{R}$ and
$I\subseteq\mathbb{N}$ we the expectation
\[
\mathbb{E}_{i\in I}\left(f(\mu^{x,i})\right)=\frac{1}{|I|}\sum_{i\in I}\int f(\mu^{x,i})\, d\mu(x).
\]
When dealing with components of several measures $\mu,\nu$, we assume
all choices of components $\mu^{x,i}$, $\nu^{y,j}$ are independent
unless otherwise stated. For instance, 
\[
\mathbb{P}_{i=n}\left(\mu^{x,i}\in A\,,\,\nu^{y,i}\in B\right)=\int\int1_{A}(\mu^{x,n})\cdot1_{B}(\nu^{y,n})\, d\mu(x)\, d\nu(y).
\]
Here $1_{A}$ is the indicator function on $A$, given by $1_{A}(\omega)=1$
if $\omega\in A$ and $0$ otherwise. 

We record one obvious fact, which we will use repeatedly: 
\begin{lem}
\label{lem:components-average-to-the-whole}For $\mu\in\mathcal{P}(\mathbb{R}^{d})$
and $n\in\mathbb{N}$,
\[
\mu=\mathbb{E}_{i=n}\left(\mu_{x,i}\right).
\]

\end{lem}
Finally, we sometimes use similar notation to average a sequence $a_{n},\ldots,a_{n+k}\in\mathbb{R}$:
\[
\mathbb{E}_{n\leq i\leq n+k}\left(a_{i}\right)=\frac{1}{k+1}\sum_{i=n}^{n+k}a_{i}.
\]

\subsection{\label{sub:inverse-theorem}An inverse theorem}

The approximate equality $H_{n}(\mu*\nu)\approx H_{n}(\mu)$ occurs
trivially if either $\mu$ is uniform (Lebesgue) measure on $[0,1]$,
or if $\nu=\delta_{x}$ is a point mass. As we saw in Section \ref{sub:Entropy-and-additive-combinatorics},
there are other ways this can occur, but the theorem below shows that
is a \emph{statistical} sense, \emph{locally }(i.e. for typical component
measures) the two trivial scenarios are essentially the only ones.
In order to state this precisely we require finite-scale and approximate
versions of being uniform and being a point mass. There are many definitions
to choose from. One possible choice is the following: 
\begin{defn}
\label{def:almost-atomic}A measure $\mu\in\mathcal{P}([0,1])$ is
$\varepsilon$\emph{-atomic} if there is an interval $I$ of length
$\varepsilon$ such that $\mu(I)>1-\varepsilon$.
\end{defn}
Alternatively we could require that the entropy be small at a given
scale, or that the random variable whose distribution is the given
measure has small variance. Up to choice of parameters these definitions
coincide and we shall use all the definitions later. See Definition
\ref{def:ep-em-almost-atomic} and the discussion following it, and
Lemma \ref{lem:concentration-from-covariance-matrix}, below.
\begin{defn}
\label{def:almost-uniform}A measure $\mu\in\mathcal{P}([0,1])$ is
\emph{$(\varepsilon,m)$-uniform if $H_{m}(\mu)>1-\varepsilon$.}
\end{defn}
Again one can imagine many alternative definitions. For example, almost-uniformity
of $\mu\in\mathcal{P}([0,1])$ at scale $\delta$ could mean that
$|\mu(I)-|I||<\delta^{2}$ for all intervals $I$ of length $|I|\geq\delta$,
or that the Fourier transform $\widehat{\mu}(\xi)$ is small at frequencies
$|\xi|<1/\delta$. Again, these definitions are essentially equivalent,
up to adjustment of parameters, to the one above. We shall not use
them here.
\begin{thm}
\label{thm:inverse-thm-Rd}For every $\varepsilon>0$ and integer
$m\geq1$ there is a $\delta=\delta(\varepsilon,m)>0$ such that for
every $n>n(\varepsilon,\delta,m)$, the following holds: if $\mu,\nu\in\mathcal{P}([0,1])$
and 
\[
H_{n}(\mu*\nu)<H_{n}(\mu)+\delta,
\]
then there are disjoint subsets $I,J\subseteq\{1,\ldots,n\}$ with
$|I\cup J|>(1-\varepsilon)n$, such that 
\begin{eqnarray*}
\mathbb{P}_{i=k}\left(\mu^{x,i}\mbox{ is }(\varepsilon,m)\mbox{-uniform}\right)\;>\;1-\varepsilon &  & \mbox{for }k\in I\\
\mathbb{P}_{i=k}\left(\nu^{x,i}\mbox{ is }\varepsilon\mbox{-atomic}\right)\;>\;1-\varepsilon &  & \mbox{for }k\in J.
\end{eqnarray*}

\end{thm}
From this it is easy to derive many variants of the theorem for the
other notions of atomicity and uniformity discussed above. In Section
\ref{sub:inverse-theorem} we give a marginally stronger statement
in which atomicity is expressed in terms of entropy.

The proof is given in Section \ref{sub:Proof-of-inverse-theorem}.
The dependence of $\delta$ on $\varepsilon,m$ is effective, but
the bounds we obtain are certainly far from optimal, and we do not
pursue this topic. The value of $n$ depends among other things on
the rate at which $H_{m}(\mu)\rightarrow\dim\mu$, which is currently
not effective.

The converse direction of the theorem is false, that is, there are
measures which satisfy the conclusion but also $H_{n}(\mu*\nu)>H_{n}(\mu)+\delta$.
To see this begin with a measure $\mu\in[0,1]$ such that $\dim(\mu*\mu)=\dim\mu=1/2$,
and such that $\lim H_{n}(\mu)=\lim H_{n}(\mu*\mu)=\frac{1}{2}$ (such
measures are not hard to construct, see e.g. \cite{ErdosVolkmann1966}
or the more elaborate constructions in \cite{Korner2008,SchmelingShmerkin2009}).
By Marstrand's theorem, for a.e. $t$ the scaled measure $\nu(A)=\mu(tA)$
satisfies $\dim\mu*\nu=1$ and hence $H_{n}(\mu*\nu)\rightarrow1$.
But it is easy to verify that, as the conclusion of the theorem holds
for the pair $\mu,\mu$, it holds for $\mu,\nu$ as well. 

Note that there is no assumption on the entropy of $\nu$, but if
$H_{n}(\nu)$ is sufficiently close to $0$ the conclusion will automatically
hold with $I$ empty, and if $H_{n}(\nu)$ is not too close to $0$
then $J$ cannot be too large relative to $n$ (see Lemma \ref{lem:entropy-local-to-global}
below). We obtain the following useful conclusion.
\begin{thm}
\label{thm:inverse-thm-R}For every $\varepsilon>0$ and integer $m$,
there is a $\delta=\delta(\varepsilon,m)>0$ such that for every $n>n(\varepsilon,\delta,m)$
and every $\mu\in\mathcal{P}([0,1])$, if 
\[
\mathbb{P}_{_{0\leq i\leq n}}\left(H_{m}(\mu^{x,i})<1-\varepsilon\right)>1-\varepsilon
\]
then for every $\nu\in\mathcal{P}([0,1])$ 
\[
H_{n}(\nu)>\varepsilon\qquad\implies\qquad H_{n}(\mu*\nu)\geq H_{n}(\mu)+\delta.
\]

\end{thm}
Specializing the above to self-convolutions we have the following
result, which shows that constructions like the one described in Section
\ref{sub:Entropy-and-additive-combinatorics} are, roughly, the only
way that $H_{n}(\mu*\mu)=H_{n}(\mu)+\delta$ can occur. This should
be compared with the results of Tao \cite{Tao2010}, who studied the
condition $H_{n}(\mu*\mu)=H_{n}(\mu)+O(\frac{1}{n})$.
\begin{thm}
\label{thm:self-convolution}For every $\varepsilon>0$ and integer
$m$, there is a $\delta=\delta(\varepsilon,m)>0$ such that for every
sufficiently large $n>n(\varepsilon,\delta,m)$ and every $\mu\in\mathcal{P}([0,1))$,
if 
\[
H_{n}(\mu*\mu)<H_{n}(\mu)+\delta
\]
then there disjoint are subsets $I,J\subseteq\{0,\ldots,n\}$ with
$|I\cup J|\geq(1-\varepsilon)n$ and such that
\begin{eqnarray*}
\mathbb{P}_{i=k}\left(\mu^{x,i}\mbox{ is }(\varepsilon,m)\mbox{-uniform}\right)\;>\;1-\varepsilon &  & \mbox{for }k\in I\\
\mathbb{P}_{i=k}\left(\mu^{x,i}\mbox{ is }\varepsilon\mbox{-atomic}\right)\;>\;1-\varepsilon &  & \mbox{for }k\in J.
\end{eqnarray*}

\end{thm}
These results hold more generally for compactly supported measures
but the parameters will depend on the diameter of the support. They
can also be extended to measures with unbounded support under additional
assumptions, see Section \ref{sub:Applications}.

\section{\label{sec:Entropy-concentration-uniformity-saturation}Entropy,
atomicity, uniformity}

\subsection{\label{sub:Preliminaries-on-entropy}Preliminaries on entropy}

The Shannon entropy of a probability measure $\mu$ with respect to
a countable partition $\mathcal{E}$ is given by
\[
H(\mu,\mathcal{E})=-\sum_{E\in\mathcal{E}}\mu(E)\log\mu(E),
\]
where the logarithm is in base $2$ and $0\log0=0$. The conditional
entropy with respect to a countable partition $\mathcal{F}$ is
\[
H(\mu,\mathcal{E}|\mathcal{F})=\sum_{F\in\mathcal{F}}\mu(F)\cdot H(\mu_{F},\mathcal{E}),
\]
where $\mu_{F}=\frac{1}{\mu(F)}\mu|_{F}$ is the conditional measure
on $F$. For a discrete probability measure $\mu$ we write $H(\mu)$
for the entropy with respect to the partition into points, and for
a probability vector $\alpha=(\alpha_{1},\ldots,\alpha_{k})$ we write
\[
H(\alpha)=-\sum\alpha_{i}\log\alpha_{i}.
\]

We collect here some standard properties of entropy.
\begin{lem}
\label{lem:entropy-combinatorial-properties}Let $\mu,\nu$ be probability
measures on a common space, $\mathcal{E},\mathcal{F}$ partitions
of the underlying space and $\alpha\in[0,1]$.
\begin{enumerate}
\item \label{enu:entropy-positivity}$H(\mu,\mathcal{E})\geq0$, with equality
if and only if $\mu$ is supported on a single atom of $\mathcal{E}$. 
\item \label{enu:entropy-combinatorial-bound}If $\mu$ is supported on
$k$ atoms of $\mathcal{E}$ then $H(\mu,\mathcal{E})\leq\log k$.
\item \label{enu:entropy-refining-partitions}If $\mathcal{F}$ refines
$\mathcal{E}$ (i.e. $\forall\; F\in\mathcal{F}\;\exists E\in\mathcal{E}\, s.t.\, F\subseteq E$)
then $H(\mu,\mathcal{F})\geq H(\mu,\mathcal{E})$.
\item \label{enu:entropy-conditional-formula}If $\mathcal{E}\lor\mathcal{F}=\{E\cap F\,:\, E\in\mathcal{E}\,,\, F\in\mathcal{F}\}$
is the join of $\mathcal{E}$ and $\mathcal{F}$, then 
\[
H(\mu,\mathcal{E}\lor\mathcal{F})=H(\mu,\mathcal{F})+H(\mu,\mathcal{E}|\mathcal{F}).
\]

\item \label{enu:entropy-concavity}$H(\cdot,\mathcal{E})$ and $H(\cdot,\mathcal{E}|\mathcal{F})$
are concave
\item \label{enu:entropy-almost-convexity}$H(\cdot,\mathcal{E})$ obeys
the ``convexity'' bound 
\[
H(\sum\alpha_{i}\mu_{i},\mathcal{E})\leq\sum\alpha_{i}H(\mu_{i},\mathcal{E})+H(\alpha).
\]

\end{enumerate}
\end{lem}
In particular, we note that for $\mu\in\mathcal{P}([0,1]^{d})$ we
have the bounds $H(\mu,\mathcal{D}_{m})\leq md$ (hence $H_{n}(\mu)\leq1$)
and $H(\mu,\mathcal{D}_{n+m}|\mathcal{D}_{n})\leq md$.

Although the function $(\mu,m)\mapsto H(\mu,\mathcal{D}_{m})$ is
not weakly continuous, the following estimates provide usable substitutes.
\begin{lem}
\label{lem:entropy-weak-continuity-properties}Let $\mu,\nu\in\mathcal{P}(\mathbb{R}^{d})$,
$\mathcal{E},\mathcal{F}$ are partitions of $\mathbb{R}^{d}$, and
$m,m'\in\mathbb{N}$. 
\begin{enumerate}
\item \label{enu:entropy-approximation} Given a compact $K\subseteq\mathbb{R}^{d}$
and $\mu\in\mathcal{P}(K)$, there is a neighborhood $U\subseteq\mathcal{P}(K)$
of $\mu$ such that $|H(\nu,\mathcal{D}_{m})-H(\mu,\mathcal{D}_{m})|=O_{d}(1)$
for $\nu\in U$.
\item \label{enu:entropy-combinatorial-distortion} If each $E\in\mathcal{E}$
intersects at most $k$ elements of $\mathcal{F}$ and vice versa,
then $|H(\mu,\mathcal{E})-H(\mu,\mathcal{F})|=O(\log k)$.
\item \label{enu:entropy-geometric-distortion} If $f,g:\mathbb{R}^{d}\rightarrow\mathbb{R}^{k}$
and $\left\Vert f(x)-g(x)\right\Vert \leq C2^{-m}$ for $x\in\mathbb{R}^{d}$
then $|H(f\mu,\mathcal{D}_{m})-H(g\mu,\mathcal{D}_{m})|\leq O_{C,k}(1)$.
\item \label{enu:entropy-translation} If $\nu(\cdot)=\mu(\cdot+x_{0})$
then $\left|H(\mu,\mathcal{D}_{m})-H(\nu,\mathcal{D}_{m})\right|=O_{d}(1)$.
\item \label{enu:entropy-change-of-scale} If $C^{-1}\leq m'/m\leq C$,
then $\left|H(\mu,\mathcal{D}_{m})-H(\mu,\mathcal{D}_{m'})\right|\leq O_{C,d}(1)$.
\end{enumerate}
\end{lem}
Recall that the total variation distance between $\mu,\nu\in\mathcal{P}(\mathbb{R}^{d})$
is
\[
\left\Vert \mu-\nu\right\Vert =\sup_{A}|\mu(A)-\nu(A)|,
\]
where the supremum is over Borel sets $A$. This is a complete metric
on $\mathcal{P}(\mathbb{R}^{d})$. It follows from standard measure
theory that for every $\varepsilon>0$ there is a $\delta>0$ such
that if $\left\Vert \mu-\nu\right\Vert <\delta$ then there are probability
measures $\tau,\mu',\nu'$ such that $\mu=(1-\varepsilon)\tau+\varepsilon\mu'$
and $\nu=(1-\varepsilon)\tau+\varepsilon\nu'$. Combining this with
Lemma \ref{lem:entropy-combinatorial-properties} \eqref{enu:entropy-concavity}
and \eqref{enu:entropy-almost-convexity}, we have
\begin{lem}
\label{lem:entropy-total-variation-continuity}For every $\varepsilon>0$
there is a $\delta>0$ such that if $\mu,\nu\in\mathcal{P}(\mathbb{R}^{d})$
and $\left\Vert \mu-\nu\right\Vert <\delta$ then for any finite partition
$\mathcal{A}$ of $\mathbb{R}^{d}$ with $k$ elements,
\[
|H(\mu,\mathcal{A})-H(\nu,\mathcal{A})|<\varepsilon\log k+H(\varepsilon)
\]
In particular, if $\mu,\nu\in\mathcal{P}([0,1]^{d})$, then
\[
|H_{m}(\mu)-H_{m}(\nu)|<\varepsilon+\frac{H(\varepsilon)}{m}
\]

\end{lem}

\subsection{\label{sub:Global-entropy-from-local-entropy}Global entropy from
local entropy}

Recall from Section \ref{sub:Component-measures} the definition of
the raw and re-scaled components $\mu_{x,n}$, $\mu^{x,n}$, and note
that 
\begin{equation}
H(\mu^{x,n},\mathcal{D}_{m})=H(\mu_{x,n},\mathcal{D}_{n+m}).\label{eq:convert-restriction-entropy-to-local-entropy}
\end{equation}
Also, note that 
\begin{eqnarray*}
\mathbb{E}_{i=n}\left(H_{m}(\mu^{x,i})\right) & = & \int\frac{1}{m}H(\mu^{x,n},\mathcal{D}_{m})\, d\mu(x)\\
 & = & \frac{1}{m}\int H(\mu_{x,n},\mathcal{D}_{n+m})\, d\mu(x)\\
 & = & \frac{1}{m}\sum_{D\in\mathcal{D}_{n}}\mu(D)H(\mu_{D},\mathcal{D}_{m+n})\\
 & = & \frac{1}{m}H(\mu,\mathcal{D}_{n+m}\,|\,\mathcal{D}_{n}).
\end{eqnarray*}

\begin{lem}
\label{lem:entropy-local-to-global}For $r\geq1$ and $\mu\in\mathcal{P}([-r,r]^{d})$
and integers $m<n$,
\begin{eqnarray*}
H_{n}(\mu) & = & \mathbb{E}_{_{0\leq i\leq n}}\left(H_{m}(\mu^{x,i})\right)+O(\frac{m}{n}+\frac{\log r}{n}).
\end{eqnarray*}
\end{lem}
\begin{proof}
By the paragraph before the lemma, the statement is equivalent to
\[
H_{n}(\mu)=\frac{1}{n}\sum_{i=0}^{n-1}\frac{1}{m}H(\mu,\mathcal{D}_{i+m}|\mathcal{D}_{i})+O(\frac{m}{n}+\frac{\log r}{n}).
\]
At the cost of adding $O(m/n)$ to the error term we can delete up
to $m$ terms from the sum. Thus without loss of generality we may
assume that $n/m\in\mathbb{N}$. When $m=1$, iterating the conditional
entropy formula and using $H(\mu,\mathcal{D}_{0})=O(\log r)$ gives
\[
\sum_{i=0}^{n-1}H(\mu,\mathcal{D}_{i+1}\,|\,\mathcal{D}_{i})=H(\mu,\mathcal{D}_{n}|\mathcal{D}_{0})=H(\mu,\mathcal{D}_{n})-O(\log r)
\]
The result follows on dividing by $n$. For general $m$, first decompose
the sum according to the residue class of $i\bmod m$ and apply the
above to each one:
\begin{eqnarray*}
\sum_{i=0}^{n-1}\frac{1}{m}H(\mu,\mathcal{D}_{i+m}\,|\,\mathcal{D}_{i}) & = & \frac{1}{m}\sum_{p=0}^{m-1}\left(\sum_{k=0}^{n/m-1}H(\mu,\mathcal{D}_{(k+1)m+p}\,|\,\mathcal{D}_{km+p})\right)\\
 & = & \frac{1}{m}\sum_{p=0}^{m-1}H(\mu,\mathcal{D}_{n+p}\,|\,\mathcal{D}_{p}).
\end{eqnarray*}
Dividing by $n$, the result follows from the bound 
\[
\left|\frac{1}{n}H(\mu,\mathcal{D}_{n+p}|\mathcal{D}_{p})-H_{n}(\mu)\right|<\frac{2m+O(\log r)}{n},
\]
which can be derived from the identities 
\begin{eqnarray*}
H(\mu,\mathcal{D}_{n})+H(\mu,\mathcal{D}_{n+p}|\mathcal{D}_{n}) & = & H(\mu,\mathcal{D}_{n+p})\\
 & = & H(\mu,\mathcal{D}_{p})+H(\mu,\mathcal{D}_{n+p}|\mathcal{D}_{p})
\end{eqnarray*}
together with the fact that $H(\mu,\mathcal{D}_{p})\leq p+\log r$
and $H(\mu,\mathcal{D}_{rm+p}|\mathcal{D}_{rm})\leq p$, and recalling
that $0\leq p<m$.
\end{proof}
We have a similar lower bound for the entropy of a convolution in
terms of convolutions of its components at each level. 
\begin{lem}
\label{lem:entropy-of-convolutions-via-component-convolutions}Let
$r>0$ and $\mu,\nu\in\mathcal{P}([-r,r]^{d})$. Then for $m<n\in\mathbb{N}$
\begin{eqnarray*}
H_{n}(\mu*\nu) & \geq & \mathbb{E}_{0\leq i\leq n}\left(\frac{1}{m}H(\mu_{x,i}*\nu_{y,i},\mathcal{D}_{i+m}|\mathcal{D}_{i})\right)+O(\frac{m+\log r}{n}).\\
 & \geq & \mathbb{E}_{0\leq i\leq n}\left(H_{m}(\mu^{x,i}*\nu^{y,i})\right)+O(\frac{1}{m}+\frac{m}{n}+\frac{\log r}{n}).
\end{eqnarray*}
\end{lem}
\begin{proof}
As in the previous proof, by introducing an error of $O(m/n)$ we
can assume that $m$ divides $n$, and by the conditional entropy
formula,
\begin{eqnarray*}
H(\mu*\nu,\mathcal{D}_{n}) & = & \sum_{k=0}^{n/m-1}H(\mu*\nu,\mathcal{D}_{(k+1)m}|\mathcal{D}_{km})+H(\mu*\nu,\mathcal{D}_{0})\\
 & = & \sum_{k=0}^{n/m-1}H(\mu*\nu,\mathcal{D}_{(k+1)m}|\mathcal{D}_{km})+O(\log r)
\end{eqnarray*}
since $\mu*\nu$ is supported on $[-2r,2r]^{d}$. Apply the linear
map $(x,y)\mapsto x+y$ to the trivial identity $\mu\times\nu=\mathbb{E}_{i=k}(\mu_{x,i}\times\nu_{y,i})$
(Lemma \eqref{lem:components-average-to-the-whole} for the product
measure). We obtain the identity $\mu*\nu=\mathbb{E}_{i=k}(\mu_{x,i}*\nu_{xy,i})$.
By concavity of conditional entropy (Lemma \ref{lem:entropy-combinatorial-properties}
\eqref{enu:entropy-concavity}),
\begin{eqnarray*}
H(\mu*\nu,\mathcal{D}_{n}) & = & \sum_{k=0}^{n/m-1}H\left(\mathbb{E}_{i=km}(\mu_{x,i}*\nu_{xy,i}),\mathcal{D}_{(k+1)m}|\mathcal{D}_{km}\right)+O(\log r)\\
 & \geq & \sum_{k=0}^{n/m-1}\mathbb{E}_{i=km}\left(H(\mu_{x,i}*\nu_{y,i},\mathcal{D}_{(k+1)m}|\mathcal{D}_{km})\right)+O(\log r),
\end{eqnarray*}
Dividing by $n$, we have shown that 
\[
H_{n}(\mu*\nu)\geq\frac{m}{n}\sum_{k=0}^{n/m-1}\mathbb{E}_{i=k}\left(H_{m}(\mu^{x,i}*\nu^{xy,i})\right)+O(\frac{m}{n}+\frac{\log r}{n}).
\]
Now do the same for the sum $k=p$ to $n/m+p$ for $p=0,1,\ldots,m-1$.
Averaging the resulting expressions gives the first inequality.

The second inequality follows from the first using
\begin{eqnarray*}
H(\mu_{x,i}*\nu_{x,i},\mathcal{D}_{(k+1)m}|\mathcal{D}_{km}) & = & H(\mu^{x,i}*\nu^{y,i},\mathcal{D}_{m}|\mathcal{D}_{0})\\
 & = & H(\mu^{x,i}*\nu^{y,i},\mathcal{D}_{m})+O(1)\\
 & = & mH_{m}(\mu^{x,i}*\nu^{y,i})+O(1),
\end{eqnarray*}
where the $O(1)$ error term arises because $\mu^{x,i}*\nu^{x,i}$
is supported on $[0,2)^{d}$ and hence meets $O(1)$ sets in $\mathcal{D}_{0}$.
\end{proof}

\subsection{Covering lemmas}

We will require some simple combinatorial lemmas.
\begin{lem}
\label{lem:covering-by-intervals}Let $I\subseteq\{0,\ldots,n\}$
and $m\in\mathbb{N}$ be given. Then there is a subset $I'\subseteq I$
such that $I\subseteq I'+[0,m]$ and $[i,i+m]\cap[j,j+m]=\emptyset$
for distinct $i,j\in I'$.\end{lem}
\begin{proof}
Define $I'$ inductively. Begin with $I'=\emptyset$ and, at each
successive stage, if $I\setminus\bigcup_{i\in I'}[i,i+m]\neq\emptyset$
then add its least element to $I'$. Stop when $I\subseteq\bigcup_{i\in I'}[i,i+m]$. \end{proof}
\begin{lem}
\label{lem:translation-invariant-covering}Let $I,J\subseteq\{0,\ldots,n\}$
and $m\in\mathbb{N}$, $\delta>0$. Suppose that $|[i,i+m]\cap J|\geq(1-\delta)m$
for $i\in I$. Then there is a subset $J'\subseteq J$ such that $|J'\cap(J'-\ell)|\geq(1-\delta-\frac{\ell}{m})|I|$
for $0\leq\ell\leq m$.\end{lem}
\begin{proof}
Let $I'\subseteq I$ be the collection obtained by applying the previous
lemma to $I$,$m$. Let $J'=J\cap(\bigcup_{i\in I'}[i,i+m])$. Then
\[
J'\cap(J'-\ell)\supseteq J\cap\bigcup_{i\in I'}([i,i+m]\cap[i-\ell,i+m-\ell])=\bigcup_{i\in I'}(J\cap[i,i+m-\ell])
\]
Also $|J\cap[i,i+m-\ell]|\geq(1-\delta-\frac{\ell}{m})m$ for $i\in I'$
, and $I\subseteq\bigcup_{i\in I'}[i,i+m]$, so by the above, 
\[
|J'\cap(J'-\ell)|\geq(1-\delta-\frac{\ell}{m})\cdot|\bigcup_{i\in I'}[i,i+m]|\geq(1-\delta-\frac{\ell}{m})|I|\qedhere
\]
\end{proof}
\begin{lem}
\label{lem:pair-of-coverings}Let $m,\delta$ be given and let $I_{1},J_{1}$
and $I_{2},J_{2}$ be two pairs of subsets of  $\{0,\ldots,n\}$ satisfying
the assumptions of the previous lemma. Suppose also that $I_{1}\cap I_{2}=\emptyset$.
Then there exist $J'_{1}\subseteq J_{1}$ and $J'_{2}\subseteq J_{2}$
with $J'_{1}\cap J'_{2}=\emptyset$ and such that $|J'_{1}\cup J'_{2}|\geq(1-\delta)^{2}|I_{1}\cup I_{2}|$.\end{lem}
\begin{proof}
Define $I'_{1}\subseteq I_{1}$ and $J'_{1}=J_{1}\cap\bigcup_{i\in I'}[i,i+m]$
as in the previous proof, so taking $\ell=0$ in its conclusion, $|J'_{1}|\geq(1-\delta)|I_{1}|$.
Let $U=\bigcup_{i\in I'_{1}}[i,i+m]$, and recall that $|J'_{1}|=|U\cap J_{1}|\geq(1-\delta)|U|$.
Since $I_{1}\subseteq U$ and $I_{1}\cap I_{2}=\emptyset$, 
\[
|J'_{1}\cap I_{2}|\leq|U|-|I_{1}|\leq\frac{1}{1-\delta}|J'_{1}|-|I_{1}|
\]
Hence, using $|J'_{1}|\geq(1-\delta)|I_{1}|$,
\begin{eqnarray*}
|J'_{1}\cup I_{2}| & = & |J'_{1}|+|I_{2}|-|J'_{1}\cap I_{2}|\\
 & \geq & |J'_{1}|+|I_{2}|-(\frac{1}{1-\delta}|J'_{1}|-|I_{1}|)\\
 & \geq & |I_{1}|+|I_{2}|-\frac{\delta}{1-\delta}|J'_{1}|\\
 & \geq & (1-\delta)|I_{1}|+|I_{2}|
\end{eqnarray*}
Now perform the analysis above with $I_{2}\setminus J'_{1},J_{2}$
in the role of $I_{1},J_{1}$ and with $J'_{1}$ in the role of $I_{2}$
(thus $(I_{2}\setminus J'_{1})\cap J'_{1}=\emptyset$ as required).
We obtain $J'_{2}\subseteq J_{2}$ such that 
\begin{eqnarray*}
|J'_{2}\cup J'_{1}| & \geq & (1-\delta)|I_{2}\setminus J'_{1}|+|J'_{1}|\\
 & = & (1-\delta)|J'_{1}\cup I_{2}|
\end{eqnarray*}
Substituting the previous bound $|J'_{1}\cup I_{2}|\geq(1-\delta)|I_{1}|+|I_{2}|$
gives the claim, except for disjointness of $J'_{1},J'_{2}$, but
clearly if they are not disjoint we can replace $J'_{1}$ with $J'_{1}\setminus J_{2}$.
\end{proof}

\subsection{Atomicity and uniformity of components}

We shall need to know almost-atomicity and almost-uniformity passes
to component measures. It will be convenient to replace the notion
of $\varepsilon$-atomic measures, introduced in Section \ref{sub:inverse-theorem},
with one that is both stronger and more convenient to work with.
\begin{defn}
\label{def:ep-em-almost-atomic}A measure $\mu\in\mathcal{P}([0,1])$
is $(\varepsilon,m)$\emph{-atomic} if $H_{m}(\mu)<\varepsilon$.
\end{defn}
Recall that $H_{m}(\mu)=0$ if and only if $\mu$ is supported on
a single interval $I\in\mathcal{D}_{m}$ of length $2^{-m}$. Thus,
by continuity of the entropy function $(p_{i})\mapsto-\sum p_{i}\log p_{i}$,
if $\varepsilon$ is small compared to $m$ then any $(\varepsilon,m)$-atomic
measure is $2^{-m}$-atomic. The reverse implication is false: indeed,
a measure may be $\varepsilon$-atomic for arbitrarily small $\varepsilon$
and at the same time have its mass divided evenly between two (adjacent)
intervals $I,I'\in\mathcal{D}_{m}$, in which case $H_{m}(\mu)=\frac{1}{m}$.
Thus, for $\varepsilon$ small compared to $m$, the most one can
say in general about an $\varepsilon$-atomic measure is that it is
$(\frac{1}{m},m)$-atomic. Thus the definition above is slightly stronger.
\begin{lem}
\label{lem:almost-atomic-measures-have-almost-atomic-components}If
$\mu\in\mathcal{P}([0,1])$ is $(\varepsilon,m)$-atomic then for
$k<m$,
\[
\mathbb{P}_{0\leq i\leq m}\left(\mu^{x,i}\mbox{ is }(\varepsilon',k)\mbox{-atomic}\right)>1-\varepsilon'
\]
for $\varepsilon'=\sqrt{\varepsilon+O(\frac{k}{m})}$. \end{lem}
\begin{proof}
By Lemma \ref{lem:entropy-local-to-global},
\[
\mathbb{E}_{0\leq i\leq m}(H_{k}(\mu^{i,x}))\leq H_{m}(\mu)+O(\frac{k}{m})<\varepsilon+O(\frac{k}{m}).
\]
Since $H_{k}(\mu^{i,x})\geq0$, the first claim follows by Markov's
inequality.\end{proof}
\begin{lem}
\label{lem:saturation-passes-to-components}If $\mu\in\mathcal{P}([0,1])$
is $(\varepsilon,n)$-uniform then for every $1\leq m<n$, 
\[
\mathbb{P}_{0\leq i\leq n}\left(\mu^{x,i}\mbox{ is }(\varepsilon',m)\mbox{-uniform}\right)>1-\varepsilon'
\]
where \textup{$\varepsilon'=\sqrt{\varepsilon+O(\frac{m}{n})}$.}\end{lem}
\begin{proof}
The proof is the same as the previous lemma and we omit it.
\end{proof}
We also will repeatedly use the following consequence of Chebychev's
inequality:
\begin{lem}
\label{lem:Chebyshev}Suppose that $\mathcal{A}\subseteq\mathcal{P}([0,1])$
and that
\[
\mathbb{P}_{0\leq i\leq n}(\mu^{x,i}\in\mathcal{A})>1-\varepsilon
\]
Then there is a subset $I\subseteq\{0,\ldots,n\}$ with $|I|>(1-\sqrt{\varepsilon})n$
and
\[
\mathbb{P}_{i=q}(\mu^{x,i}\in\mathcal{A})>1-\sqrt{\varepsilon}\qquad\mbox{for }q\in I
\]
\end{lem}
\begin{proof}
Consider the function $f:\{0,\ldots,m\}\rightarrow[0,1]$ given by
$f(q)=\mathbb{P}_{i=q}(\mu^{x,i}\in\mathcal{A})$. By assumption $\mathbb{E}_{0\leq q\leq n}(f(q))>1-\varepsilon$.
By Chebychev's inequality, there is a subset $I\subseteq\{0,\ldots,n\}$
with $|I|\geq(1-\sqrt{\varepsilon})n$ and $f(q)>1-\sqrt{\varepsilon}$
for $q\in I$, as desired.
\end{proof}

\section{\label{sec:Entropy-growth-for-convolutions}Convolutions}

\subsection{\label{sub:Covariance-matrices}\label{sub:Normal-measures-and-Berry-Esseen}The
Berry-Esseen theorem and an entropy estimate}

For $\mu\in\mathcal{P}(\mathbb{R})$ let $m(\mu)$ denote the mean,
or barycenter, of $\mu$, given by 
\[
\left\langle \mu\right\rangle =\int x\, d\mu(x),
\]
and let $\var(\mu)$ denote its variance:
\[
\var(\mu)=\int(x-\left\langle \mu\right\rangle )^{2}\, d\mu(x).
\]
Recall that if $\mu_{1},\ldots,\mu_{k}\in\mathcal{P}(\mathbb{R})$
then $\mu=\mu_{1}*\ldots*\mu_{k}$ has mean $\left\langle \mu\right\rangle =\sum_{i=1}^{k}\left\langle \mu_{i}\right\rangle $
and $\var(\mu)=\sum_{i=1}^{k}\var(\mu_{i})$. 

The Gaussian with mean $m$ and variance $\sigma^{2}$ is given by
$\gamma_{m,\sigma^{2}}(A)=\int_{A}\varphi((x-m)/\sigma^{2})dx$, where
$\varphi(x)=\sqrt{2\pi}\exp(-\frac{1}{2}|x|^{2})$. The central limit
theorem asserts that, for $\mu_{1},\mu_{2},\ldots\in\mathcal{P}(\mathbb{R}^{d})$
of positive variance, the convolutions $\mu_{1}*\ldots*\mu_{k}$ can
be re-scaled so that the resulting measure is close in the weak sense
to a Gaussian measure. The Berry-Esseen inequalities quantify the
rate of this convergence. We use the following variant from \cite{Esseen1942}.
\begin{thm}
\label{thm:Berry-Esseen-Rotar}Let $\mu_{1},\ldots,\mu_{k}$ be probability
measures on $\mathbb{R}$ with finite third moments $\rho_{i}=\int|x|^{3}\, d\mu_{i}(x)$.
Let $\mu=\mu_{1}*\ldots*\mu_{k}$, and let $\gamma$ be the Gaussian
measure with the same mean and variance as $\mu$. Then%
\footnote{In the usual formulation one considers the measure $\mu'$ defined
by scaling $\mu$ by $\var(\mu)$, and $\gamma'$ the Gaussian with
the same mean and variance $1=\var(\mu')$, and gives a similar bound
for $|\mu'(J)-\gamma'(J)|$ as $J$ ranges over intervals. The two
formulations are equivalent since $\mu(I)-\gamma(I)=\mu'(J)-\gamma'(J)$
where $J$ is an interval depending in the obvious manner on $I$,
and $I\rightarrow J$ is a bijection.%
} for any interval $I\subseteq\mathbb{R}$,
\[
|\mu(I)-\gamma(I)|\leq C_{1}\cdot\frac{\sum_{i=1}^{k}\rho_{i}}{\var(\mu)^{3/2}},
\]
where $C_{1}=C_{1}(d)$. In particular, if $\rho_{i}\leq C$ and $\sum_{i=1}^{k}\var(\mu_{i})\geq ck$
for constants $c,C>0$ then 
\[
|\mu(I)-\gamma(I)|=O_{c,C}(k^{-1/2}).
\]

\end{thm}

\subsection{\label{sub:Estimating-modulus-of-continuity}Multiscale analysis
of repeated self-convolutions}

In this section we show that for any measure $\mu$, every $\delta>0$,
every integer scale $m\geq2$, and appropriately large $k$, the following
holds: typical levels-$i$ components of the convolution $\mu^{*k}$
are $(\delta,m)$-uniform, unless in $\mu$ the level-$i$ components
are typically $(\delta,m)$-atomic. The main idea is to apply the
Berry-Esseen theorem to convolutions of component measures.
\begin{prop}
\label{pro:entropy-convolution-estimate}Let $\sigma>0$, $\delta>0$,
and $m\geq2$ an integer. Then there exists an integer $p=p_{0}(\sigma,\delta,m)$
such that for all $k\geq k_{0}(\sigma,\delta,m)$, the following holds:

Let $\mu_{1},\ldots,\mu_{k}\in\mathcal{P}([0,1])$, let $\mu=\mu_{1}*\ldots*\mu_{k}$,
and suppose that $\var(\mu)\geq\sigma k$. Then 
\begin{equation}
\mathbb{P}_{i=p-[\log\sqrt{k}]}\left(\mu^{x,i}\mbox{ is }(\delta,m)\mbox{-uniform}\right)>1-\delta.\label{eq:3}
\end{equation}

\end{prop}
Note that $p-[\log\sqrt{k}]$ will generally be negative. Dyadic partitions
of level $q$ with $q<0$ are defined in the same manner as for positive
$q$, that is by $\mathcal{D}_{q}=\{[r2^{q},(r+1)2^{q})\}_{r\in\mathbb{Z}}$.
For $q<0$ this partition consists of intervals of length is $2^{|q|}$
with integer endpoints. Thus, the conclusion of the proposition concerns
the $\mu$-probabilities of nearby intervals of length $O_{p}(\sqrt{k})=O_{\sigma,\delta,m}(\sqrt{k})$
(since $p=p_{0}(\sigma,\delta,m)$). This is the natural scale at
which we can expect to control such probabilities: indeed, $\mu$
is close to a Gaussian $\gamma$ of variance $\sigma k$, but only
in the sense that for any $c$, if $k$ is large enough, $\mu$ and
$\gamma$ closely agree on the mass that they give to intervals of
length $c\sqrt{\var(\mu)}=c\sqrt{k}$. 
\begin{proof}
Let us first make some elementary observations. Suppose that $\gamma\in\mathcal{P}(\mathbb{R})$
is a probability measure with continuous density function $f$, and
$x\in\mathbb{R}$ is such that $f(x)\neq0$. Since $\gamma(I)=\int_{I}f(y)dy$,
for any interval $I$ we have
\[
\left|\frac{\gamma(I)}{|I|}-f(x)\right|\leq\sup_{z\in I}|f(x)-f(z)|
\]
where $|I|$ is the length of $I$. By continuity, the right hand
side tends to $0$ uniformly as the endpoints of $I$ approach $x$.
In particular, if $n$ is large enough, for any $I\subseteq\mathcal{D}_{n}(x)$
the ratio $\frac{\gamma(x)}{|I|}$ will be arbitrarily close to $f(x)$.
Therefore, since $f(x)\neq0$, for any fixed $m$, if $n$ is large
enough then $|\frac{\gamma(I)}{\gamma(J)}-1|=|\frac{\gamma(I)/|I|}{\gamma(J)/|J|}-1|$
for all intervals $I,J\in\mathcal{D}_{n+m}$ with $I,J\subseteq\mathcal{D}_{n}(x)$.
In other words, the distribution of $\gamma^{x,n}$ on the level-$m$
dyadic subintervals of $[0,1)$ approaches the uniform one as $n\rightarrow\infty$.
Now, 
\[
H_{m}(\mu^{x,n})=-\sum_{I\in\mathcal{D}_{n+m},I\subseteq\mathcal{D}_{n}(x)}\mu(I)\log\mu(I),
\]
and the function $t\log t$ is continuous for $t\in(0,1)$. Therefore,
 writing $u$ for the uniform measure on $[0,1)$, we conclude that
\[
\lim_{n\rightarrow\infty}H_{m}(\gamma^{x,n})=H_{m}(u)=1.
\]
This in turn implies that $\mathbb{E}_{i=p}(H_{m}(\gamma^{x,p}))\rightarrow1$
as $p\rightarrow\infty$. Finally, the rate of convergence in the
limits above is easily seen to depend only on the value $f(x)$ and
the modulus of continuity of $f$ at $x$. 

Fix $0<\sigma,\delta<1$ and consider the family $\mathcal{G}$ of
Gaussians with mean 0 and variance in the interval $[\sigma,1]$.
For every interval $I=[-R,R]$, the restriction to $I$ of the density
functions of measures in $\mathcal{G}$ form an equicontinuous family.
Also, by choosing a large enough $R$ we can ensure that $\inf_{g\in\mathcal{G}}\gamma([-R,R])$
is arbitrarily close to $1$. Therefore, by the previous discussion,
there is a $p=p_{0}(\sigma,\delta,m)$ such that $\mathbb{P}_{i=p}(H_{m}(\gamma^{x,i})>1-\delta)>1-\delta$
for all $\gamma\in\mathcal{G}$.

Now, if $\mu_{i}$ and $\mu$ are as in the statement and $\mu'$
is $\mu$ scaled by $2^{-[\log\sqrt{k}]}$ (which is up to a constant
factor the same as $1/\sqrt{k}$), then by the Berry-Esseen theorem
(Theorem \ref{thm:Berry-Esseen-Rotar}) $\mu'$ agrees with the Gaussian
of the same mean and variance on intervals of length $2^{-p-m}$ to
a degree that can be made arbitrarily small by making $k$ large in
a manner depending on $\sigma,p$. In particular for large enough
$k$ this guarantees that $\mathbb{P}_{i=p}(H_{m}((\mu')^{x,i})>1-\delta)>1-\delta$. 

All that remains is to adjust the scale by a factor of $2^{[\log\sqrt{k}]}$.
Then the same argument applied to $\mu$ instead of the scaled $\mu'$
gives $\mathbb{P}_{i=p-[\log\sqrt{k}]}(H_{m}((\mu)^{x,i})>1-\delta)>1-\delta$,
which is \eqref{eq:3}.
\end{proof}
We turn to repeated self-convolutions.
\begin{prop}
\label{prop:saturation-of-components-of-convolution}Let $\sigma,\delta>0$
and $m\geq2$ an integer. Then there exists $p=p_{1}(\sigma,\delta,m)$
such that for sufficiently large $k\geq k_{1}(\sigma,\delta,m)$,
the following holds. 

Let $\mu\in\mathcal{P}([0,1])$, fix an integer $i_{0}\geq0$, and
write 
\[
\lambda=\mathbb{E}_{i=i_{0}}\left(\var(\mu^{x,i})\right).
\]
If $\lambda>\sigma$ then for $j_{0}=i_{0}-[\log\sqrt{k}]+p$ and
$\nu=\mu^{*k}$ we have 
\[
\mathbb{P}_{j=j_{0}}\left(\nu^{x,j}\mbox{ is }(\delta,m)\mbox{-uniform}\right)>1-\delta.
\]
\end{prop}
\begin{proof}
Let $\mu$, $\lambda$ and $m$ be given. Fix $p$ and $k$ (we will
later see how large they must be). Let $i_{0}$ be as in the statement
and $j_{0}=i_{0}-[\log\sqrt{k}]+p$.

Let $\widetilde{\mu}$ denote the $k$-fold self-product $\widetilde{\mu}=\mu\times\ldots\times\mu$
and let $\pi:(\mathbb{R})^{k}\rightarrow\mathbb{R}$ denote the addition
map 
\[
\pi(x_{1},\ldots,x_{k})=\sum_{i=1}^{k}x_{i}.
\]
Then $\nu=\pi\widetilde{\mu}$, and, since $\widetilde{\mu}=\mathbb{E}_{i=i_{0}}\left(\widetilde{\mu}_{x,i}\right)$,
we also have by linearity $\nu=\mathbb{E}_{i=i_{0}}\left(\pi\widetilde{\mu}_{x,i}\right)$.
By concavity of entropy and an application of Markov's inequality,
there is a $\delta_{1}>0$, depending only on $\delta$, such that
the proposition will follow if we show that with probability $>1-\delta_{1}$
over the choice of the component $\widetilde{\mu}_{x,i_{0}}$ of $\widetilde{\mu}$,
the measure $\eta=\pi\widetilde{\mu}_{x,i_{0}}$ satisfies
\begin{equation}
\mathbb{P}_{j=j_{0}}\left(\eta^{y,j}\mbox{ is }(\delta_{1},m)\mbox{-uniform}\right)>1-\delta_{1}.\label{eq:5}
\end{equation}

The random component $\widetilde{\mu}_{x,i_{0}}$ is itself a product
measure $\widetilde{\mu}_{x,i}=\mu_{x_{1},i_{0}}\times\ldots\times\mu_{x_{k},i_{0}}$,
and the marginal measures $\mu_{x_{j},i_{0}}$ of this product are
distributed independently according to the distribution of the raw
components of $\mu$ at level $i_{0}$. Note that these components
differ from the re-scaled components by a scaling factor of $2^{i_{0}}$,
so the expected variance of the raw components is $2^{-2i_{0}}\lambda$.
Recall that 
\[
\var(\pi(\mu_{x_{1},i_{0}}\times\ldots\times\mu_{x_{k},i_{0}}))=\sum_{j=1}^{k}\var(\mu_{x_{j},i_{0}}).
\]
Thus for any $\delta_{2}>0$, by the weak law of large numbers, if
$k$ is large enough in a manner depending on $\delta_{2}$ then with
probability $>1-\delta_{2}$ over the choice of $\widetilde{\mu}_{x,i_{0}}$
we will have%
\footnote{We use here the fact that we have a uniform bound for the rate of
convergence in the weak law of large numbers for i.i.d. random variables
$X_{1},X_{2},\ldots$. In fact, the rate can be bounded in terms of
the mean and variance of $X_{1}$. Here $X_{1}$ is distributed like
the variance $\var(\mu_{x,i_{0}})$ of a random component of level
$i_{0}$, and the mean and variance of $X_{1}$ are bounded independently
of $\mu\in\mathcal{P}([0,1])$.%
}
\begin{equation}
|\frac{1}{k}\var(\pi\widetilde{\mu}_{x,i_{0}})-2^{-2i_{0}}\lambda|<2^{-2i_{0}}\delta_{2}.\label{eq:4}
\end{equation}
We can choose $\delta_{2}$ small in a manner depending on $\sigma$,
so \eqref{eq:4} implies 
\begin{eqnarray}
\var(\pi\widetilde{\mu}_{x,i_{0}}) & > & 2^{-2i_{0}}\cdot k\sigma/2.\label{eq:6}
\end{eqnarray}
But now inequality \eqref{eq:5} follows from an application of Proposition
\ref{pro:entropy-convolution-estimate} with proper choice of parameters.\end{proof}
\begin{lem}
\label{lem:concentration-from-covariance-matrix}Fix $m\in\mathbb{N}$.
If $\var(\mu)$ is small enough then $H_{m}(\mu)\leq\frac{2}{m}$.
If $H_{m}(\mu)$ is small enough then $\var(\mu)<2^{-m}$.\end{lem}
\begin{proof}
If $\var(\mu)$ is small then most of the $\mu$-mass sits on an interval
of length $2^{-m}$, hence on at most two intervals from $\mathcal{D}_{m}$,
so $H_{m}(\mu)$ is roughly $\frac{1}{m}$ (certainly $<\frac{2}{m}$).
Conversely, if $H_{m}(\mu)$ is small then most of the $\mu$-mass
sits on one interval from $\mathcal{D}_{m}$, whose length is $2^{-m}$,
so $\var(\mu)$ is of this order.
\end{proof}
Recall Definitions \ref{def:almost-uniform} and \ref{def:ep-em-almost-atomic}.
\begin{cor}
\label{cor:components-of-measures-with-small-variance}Let $m\in\mathbb{N}$
and $\varepsilon>0$. For $N>N(m,\varepsilon)$ and $0<\delta<\delta(m,\varepsilon,N)$,
if $\mu\in\mathbb{P}([0,1])$ and $\var(\mu)<\delta$, then
\[
\mathbb{P}_{0\leq i\leq N}(\var(\mu^{x,i})<\varepsilon\mbox{ and }\mu^{x,i}\mbox{ is }(\varepsilon,m)\mbox{-atomic})>1-\varepsilon
\]
\end{cor}
\begin{proof}
Using the previous lemma choose $m',\varepsilon'$ such that $H_{m'}(\theta)<\varepsilon'$
implies $\var(\theta)<\varepsilon$. Then it suffices to find $\delta,N$
such that
\[
\mathbb{P}_{0\leq i\leq N}(H_{m'}(\mu^{x,i})<\varepsilon'\mbox{ and }H_{m}(\mu^{x,i})<\varepsilon)>1-\varepsilon
\]
By Lemma \ref{lem:almost-atomic-measures-have-almost-atomic-components}
(applied twice), if $\varepsilon''>0$ is small enough then for large
enough $N$ the last inequality follows from $H_{N}(\mu)<\varepsilon''$.
Finally, by the last lemma again, if $N$ is large enough, this follows
from $\var(\mu)<\delta$ if $\delta$ is sufficiently small.\end{proof}
\begin{thm}
\label{thm:saturation-of-repeated-convolutions}Let $\delta>0$ and
$m\geq2$. Then for $k\geq k_{2}(\delta,m)$ and all sufficiently
large $n\geq n_{2}(\delta,m,k)$, the following holds:

For any $\mu\in\mathcal{P}([0,1])$ there are disjoint subsets $I,J\subseteq\{1,\ldots,n\}$
with $|I\cup J|>(1-\delta)n$ such that, writing $\nu=\mu^{*k}$,
\begin{eqnarray}
\mathbb{P}_{i=q}\left(\nu^{x,i}\mbox{ is }(\delta,m)\mbox{-uniform}\right)\geq1-\delta &  & \mbox{ for }q\in I\label{eq:73}\\
\mathbb{P}_{i=q}\left(\mu^{x,i}\mbox{ is }(\delta,m)\mbox{-atomic}\right)\geq1-\delta &  & \mbox{ for }q\in J.\label{eq:74}
\end{eqnarray}
\end{thm}
\begin{proof}
Let $\delta$ and $m\geq0$ be given, we may assume $\delta<1/2$. 

The proof is given in terms of a function $\widetilde{\rho}:(0,1]\rightarrow(0,1]$
with $\widetilde{\rho}(\sigma)$ depending on $\sigma,\delta,m$.
The exact requirements will be given in the course of the proof. The
definition of $\widetilde{\rho}$ uses the functions $k_{1}(\cdot)$
and $p_{1}(\cdot)$ from Proposition \ref{prop:saturation-of-components-of-convolution}
and we assume, without loss of generality, that these functions are
monotone in each of their arguments. 

Our first requirement of $\widetilde{\rho}$ will be that $\widetilde{\rho}(\sigma)<\sigma$.
Consider the decreasing sequence $\sigma_{0}>\sigma_{1}>\ldots$ defined
by $\sigma_{0}=1$ and $\sigma_{i}=\widetilde{\rho}(\sigma_{i-1})$.
Assume that $k\geq k_{1}(\sigma_{\left\lceil 1+2/\delta\right\rceil },\delta,m)$;
this expression can be taken for $k_{2}(\delta,m)$.

Fix $\mu$ and $n$ large, we shall later see how large an $n$ is
desirable. For $0\leq q\leq n$ write
\[
\lambda_{q}=\mathbb{E}_{i=q}\left(\var(\mu^{x,i})\right).
\]
Since the intervals $(\sigma_{i},\sigma_{i-1}]$ are disjoint, there
is an integer $1\leq s\leq1+\frac{2}{\delta}$ such that $\mathbb{P}_{0\leq q\leq n}(\lambda_{q}\in(\sigma_{s},\sigma_{s-1}])<\frac{\delta}{2}$.
For this $s$ define 
\begin{eqnarray*}
\sigma & = & \sigma_{s-1}\\
\rho & = & \widetilde{\rho}(\sigma)\;=\;\sigma_{s},
\end{eqnarray*}
and set
\begin{eqnarray*}
I' & = & \{0\leq q\leq n\,:\,\lambda_{q}>\sigma\}\\
J' & = & \{0\leq q\leq n\,:\,\lambda_{q}<\rho\}.
\end{eqnarray*}
Then by our choice of $s$, 
\begin{equation}
|I'\cup J'|>(1-\frac{\delta}{2})n.\label{eq:71}
\end{equation}

Let $\ell\geq0$ be the integer
\[
\ell=[\log\sqrt{k}]-p_{1}(\sigma,\delta,m).
\]
Since we may take $n$ large relative to $\ell$, by deleting at most
$\ell$ elements of $I'$ we can assume that $I'\subseteq[\ell,n]$
and that \eqref{eq:71} remain valid. Let 
\[
I=I'-\ell
\]
Since $k\geq k_{1}(\sigma,\delta,m)$, by our choice of parameters
and the previous proposition, 
\[
\mathbb{P}_{i=q}\left(\nu^{x,i}\mbox{ is }(\delta,m)\mbox{-uniform}\right)>1-\delta\qquad\mbox{for }q\in I,
\]
which is \eqref{eq:73}. 

We now turn to the slightly harder task of choosing $n$ (i.e. determining
the appropriate condition $n\geq n_{2}$). By definition of $J'$,
\[
\mathbb{E}_{i=q}\left(\var(\mu^{x,i})\right)=\lambda_{q}<\rho\qquad\mbox{for }q\in J'.
\]
This and Markov's inequality imply 
\begin{equation}
\mathbb{P}_{i=q}\left(\var(\mu^{x,i})<\sqrt{\rho}\right)>1-\sqrt{\rho}\qquad\mbox{for }q\in J'.\label{eq:76}
\end{equation}
Fix a small number $\rho'=\rho'(\delta,\sigma)$ and a large integer
$N=N(\ell,\delta,\rho')$ upon which we place constraints in due course.
Since we can take $n$ large relative to $N$, we can assume $I',J'\subseteq\{\ell,\ldots,n-N\}$
without affecting the size bounds. Assuming $\rho$ is small enough,
Corollary \ref{cor:components-of-measures-with-small-variance} tells
us that any measure $\theta\in\mathcal{P}([0,1])$ satisfying $\var(\theta)<\sqrt{\rho}$
also satisfies
\[
\mathbb{P}_{0\leq i\leq N}\left(\var(\theta^{y,i})<\sigma\mbox{ and }\theta^{y,i}\mbox{ is }(\delta,m)\mbox{-atomic}\right)>1-\rho'
\]
Assuming again that $\sqrt{\rho}<\rho'$, the last equation and \eqref{eq:76}
give 
\begin{eqnarray*}
\mathbb{P}_{q\leq i\leq q+N}\left(\var(\mu^{x,i})<\sigma\mbox{ and }\mu^{x,i}\mbox{ is }(\delta,m)\mbox{-atomic}\right) & > & (1-\sqrt{\rho})(1-\rho')\\
 & > & 1-2\rho'.\qquad\mbox{for }q\in J'
\end{eqnarray*}
Let
\[
U=\left\{ q\in\mathbb{N}\,:\,\mathbb{P}_{i=q}(\var(\theta^{y,i})<\frac{\sigma}{2}\mbox{ and }\theta^{y,i}\mbox{ is }(\delta,m)\mbox{-atomic})>1-\sqrt{2\rho'}\right\} .
\]
By Lemma \ref{lem:Chebyshev} (i.e. Chebychev's inequality), 
\[
|U\cap[q,q+N]|\geq(1-\sqrt{2\rho'})N\qquad\mbox{for }q\in J'.
\]
Apply Lemma \ref{lem:translation-invariant-covering} to $J'$ and
$U$ to obtain $U'\subseteq U$ satisfying $|U'|>(1-\sqrt{2\rho'})|J'|$
and $|U'\cap(U'-\ell)|>(1-2\sqrt{2\rho'}-\frac{\ell}{N})|U'|$. Defining
\[
J=U'\cap(U'-\ell)
\]
and assuming  that $\frac{\ell}{N}<2\sqrt{\rho'}$ we conclude that
\[
|J|\geq(1-3\sqrt{2\rho'})|J'|
\]

We claim that $I\cap J=\emptyset$. Indeed, suppose $q\in I\cap J$.
Then $q+\ell\in I'$, so $\lambda_{q+\ell}\geq\sigma$. On the other
hand, $q\in J\subseteq U'-\ell$ implies $q+\ell\in U'\subseteq U$,
so by definition of $U$ and assuming  that $3\sqrt{3\rho'}<\sigma$,
\begin{eqnarray*}
\lambda_{q+\ell} & = & \mathbb{E}_{i=q+\ell}(\var(\mu^{x,i}))\\
 & \leq & \frac{\sigma}{2}\cdot\mathbb{P}_{i=q+\ell}(\var(\mu^{x,i})<\frac{\sigma}{2})+1\cdot\mathbb{P}_{i=q+\ell}(\var(\mu^{x,i})\geq\frac{1}{2})\\
 & < & \frac{\sigma}{2}\cdot1+1\cdot3\sqrt{3\rho'}\\
 & < & \sigma.
\end{eqnarray*}
This contradiction shows that $I\cap J=\emptyset$. 

Finally, $I'\cap J'=\emptyset$ and $|I'\cup J'|>(1-\frac{\delta}{2})n$,
so, assuming  that $3\sqrt{3\rho'}<\delta$, 
\[
|I\cup J|=|I|+|J|\geq|I|+(1-3\sqrt{3\rho'})|J'|>(1-\frac{\delta}{2})|I'\cup J'|>(1-\frac{\delta}{2})^{2}n.
\]
This completes the proof.
\end{proof}

\subsection{\label{sub:Kaimanovitch-Vershik-Tao-theorem}The Ka\u\i{}manovich-Vershik
lemma}

The Pl\"{u}nnecke-Rusza inequality in additive combinatorics toughly
states that if $A,B\subseteq\mathbb{Z}$ and $|A+B|\leq C|A|$, then
there is a subset $A_{0}\subseteq A$ of size comparable to $A$ such
that $|A_{0}+B^{\oplus k}|\leq C^{k}|A|$. The second ingredient in
our proof of Theorem \ref{thm:inverse-thm-Rd} is the following elegant
analog for entropy:
\begin{lem}
\label{thm:Kaimanovitch-Vershik-Tao} Let $\Gamma$ be a countable
abelian group and let $\mu,\nu\in\mathcal{P}(\Gamma)$ be probability
measures with $H(\mu)<\infty$, $H(\nu)<\infty$. Let 
\[
\delta_{k}=H(\mu*(\nu^{*(k+1)}))-H(\mu*(\nu^{*k})).
\]
Then $\delta_{k}$ is non-increasing in $k$. In particular, 
\[
H(\mu*(\nu^{*k}))\leq H(\mu)+k\cdot(H(\mu*\nu)-H(\nu)).
\]

\end{lem}
This lemma above first appears in a study of random walks on groups
by Ka\u\i{}manovich and Vershik \cite{KaimanovichVershik1983}. It
was more recently rediscovered and applied in additive combinatorics
by Madiman and his co-authors  \cite{Madiman2008,MadimanMarcusTetali2012}
and, in a weaker form, by Tao \cite{Tao2010}, who later made the
connection to additive combinatorics. For completeness we give the
short proof here.
\begin{proof}
Let $X_{0}$ be a random variable distributed according to $\mu$,
let $Z_{n}$ be distributed according to $\nu$, and let all variables
be independent. Set $X_{n}=X_{0}+Z_{1}+\ldots+Z_{n}$, so the distribution
of $X_{n}$ is just $\mu*\nu^{*n}$. Furthermore, since $G$ is abelian,
given $Z_{1}=g$, the distribution of $X_{n}$ is the same as the
distribution of $X_{n-1}+g$ and hence $H(X_{n}|Z_{1})=H(X_{n-1})$.
We now compute: 
\begin{eqnarray}
H(Z_{1}|X_{n}) & = & H(Z_{1},X_{n})-H(X_{n})\nonumber \\
 & = & H(Z_{1})+H(X_{n}|Z_{1})-H(X_{n})\nonumber \\
 & = & H(\nu)+H(\mu*\nu^{*(n-1)})-H(\mu*\nu^{*n}).\label{eq:9}
\end{eqnarray}
Since $X_{n}$ is a Markov process, given $X_{n}$, $Z_{1}=X_{1}-X_{0}$
is independent of $X_{n+1}$, so 
\[
H(Z_{1}\,|\, X_{n})=H(Z_{1}\,|\, X_{n},X_{n+1})\leq H(Z_{1}\,|\, X_{n+1}).
\]
Using \eqref{eq:9} in both sides of the inequality above, we find
that
\[
H(\mu*\nu^{*(n-1)})-H(\mu*\nu^{*n})\leq H(\mu*\nu^{*n})-H(\mu*\nu^{*(n+1)}),
\]
which is the what we claimed.
\end{proof}
For the analogous statement for the scale-$n$ entropy of measures
on $\mathbb{R}$ we use a discretization argument. For $m\in\mathbb{N}$
let 
\[
M_{m}=\{\frac{k}{2^{m}}\,:\, k\in\mathbb{Z}\}
\]
denote the group of $2^{m}$-adic rationals. Each $D\in\mathcal{D}_{m}$
contains exactly one $x\in M_{m}$. Define the $m$-discretization\emph{
}map $\sigma_{m}:\mathbb{R}\rightarrow M_{m}$ by $\sigma_{m}(x)=v$
if $\mathcal{D}_{m}(x)=\mathcal{D}_{m}(v)$, so that $\sigma_{m}(x)\in\mathcal{D}_{m}(x)$.

We say that a measure $\mu\in\mathcal{P}(\mathbb{R}^{d})$ is $m$-discrete
if it is supported on $M_{m}$. For arbitrary $\mu$ its $m$-discretization
is its push-forward $\sigma_{m}\mu$ through $\sigma_{m}$, given
explicitly by: 
\[
\sigma_{m}\mu=\sum_{v\in M_{m}^{d}}\mu(\mathcal{D}_{m}(v))\cdot\delta_{v}.
\]
Clearly $H_{m}(\mu)=H_{m}(\sigma_{m}\mu)$.
\begin{lem}
\label{lem:entropy-of-discretized-convolutions}Given $\mu_{1},\ldots,\mu_{k}\in\mathcal{P}(\mathbb{R})$
with $H(\mu_{i})<\infty$ and $m\in\mathbb{N}$,
\[
|H_{m}(\mu_{1}*\mu_{2}*\ldots*\mu_{k})-H_{m}(\sigma_{m}\mu_{1}*\ldots*\sigma_{m}\mu_{k})|=O(k/m).
\]
\end{lem}
\begin{proof}
Let $\pi:\mathbb{R}^{k}\rightarrow\mathbb{R}$ denote the map $(x_{1},\ldots,x_{k})\mapsto\sum_{i=1}^{k}x_{i}$.
Then $\mu_{1}*\ldots*\mu_{k}=\pi(\mu_{1}\times\ldots\times\mu_{k})$
and $\mu_{1}^{(m)}*\ldots*\mu_{k}^{(m)}=\pi\circ\sigma_{m}^{k}(\mu_{1}\times\ldots\times\mu_{k})$
(here $\sigma_{m}^{k}:(x_{1},\ldots,x_{k})\mapsto(\sigma_{m}x_{1},\ldots,\sigma_{m}x_{k})$).
Now, it is easy to check that 
\[
|\pi(x_{1},\ldots,x_{k})-\pi\circ\sigma_{m}^{k}(x_{1},\ldots,x_{k})|=O(k)
\]
so the desired entropy bound follows from Lemma \ref{lem:entropy-weak-continuity-properties}
\eqref{enu:entropy-geometric-distortion}. \end{proof}
\begin{prop}
\label{cor:non-discrete-KVT-theorem}Let $\mu,\nu\in\mathcal{P}(\mathbb{R})$
with $H_{n}(\mu),H_{n}(\nu)<\infty$. Then
\begin{equation}
H_{n}(\mu*(\nu^{*k}))\leq H_{n}(\mu)+k\cdot\left(H_{n}(\mu*\nu)-H_{n}(\mu)\right)+O(\frac{k}{n}).\label{eq:10}
\end{equation}
\end{prop}
\begin{proof}
Writing $\widetilde{\mu}=\sigma_{n}(\mu)$ and $\widetilde{\nu}=\sigma_{n}(\nu)$,
Theorem \ref{thm:Kaimanovitch-Vershik-Tao} implies 
\[
H(\widetilde{\mu}*(\widetilde{\nu}^{*k}))\leq H(\widetilde{\mu})+k\cdot(H(\widetilde{\mu}*\widetilde{\nu})-H(\widetilde{\nu})).
\]
For $n$-discrete measures the entropy of the measure coincides with
its entropy with respect to $\mathcal{D}_{n}$, so dividing this inequality
by $n$ gives \eqref{eq:10} for $\widetilde{\mu},\widetilde{\nu}$
instead of $\mu,\nu$, and without the error term. The desired inequality
follows from Lemma \ref{lem:entropy-of-discretized-convolutions}.
\end{proof}
We also will later need the following simple fact:
\begin{cor}
\label{lem:entropy-monotonicity-under-convolution}For $m\in\mathbb{N}$
and $\mu,\nu\in\mathcal{P}([-r,r]^{d})$ with $H_{n}(\mu),H_{n}(\nu)<\infty$,
\[
H_{m}(\mu*\nu)\geq H_{m}(\mu)-O(\frac{1}{m}).
\]
\end{cor}
\begin{proof}
This is immediate from the identity $\mu*\nu=\int\mu*\delta_{y}\, d\nu(y)$,
concavity of entropy, and Lemma \ref{lem:entropy-weak-continuity-properties}
\eqref{enu:entropy-translation} (note that $\mu*\delta_{y}$ is a
translate of $\mu$). 
\end{proof}

\subsection{\label{sub:Proof-of-inverse-theorem}Proof of the inverse theorem}

Recall Definitions \ref{def:almost-uniform} and \ref{def:ep-em-almost-atomic}.
\begin{thm}
For every $\varepsilon_{1},\varepsilon_{2}>0$ and integers $m_{1},m_{2}\geq2$,
there exists a $\delta=\delta(\varepsilon_{1},\varepsilon_{2},m_{1},m_{2})$
such that for all $n>n(\varepsilon_{1},\varepsilon_{2},m_{1},m_{2},\delta)$,
if $\nu,\mu\in\mathcal{P}([0,1])$ then either $H_{n}(\mu*\nu)\geq H_{n}(\mu)+\delta$,
or there exist disjoint subsets $I,J\subseteq\{0,\ldots,n\}$ with
$|I\cup J|\geq(1-\varepsilon)n$ and
\begin{eqnarray*}
\mathbb{P}_{i=k}\left(\mu^{x,i}\mbox{ is }(\varepsilon_{1},m_{1})\mbox{-uniform}\right)\;>\;1-\varepsilon &  & \mbox{for }k\in I\\
\mathbb{P}_{i=k}\left(\nu^{x,i}\mbox{ is }(\varepsilon_{2},m_{2})\mbox{-atomic}\right)\;>\;1-\varepsilon &  & \mbox{for }k\in J.
\end{eqnarray*}
\end{thm}
\begin{rem*}
Since, given $\varepsilon$, for a suitable choice of $\varepsilon_{2},m_{2}$
any $(\varepsilon',m')$-atomic measure is $\varepsilon_{1}$-atomic,
the statement above implies Theorem \ref{thm:inverse-thm-Rd}.\end{rem*}
\begin{proof}
We begin with $\varepsilon_{1}=\varepsilon_{2}=\varepsilon$ and $m_{1}=m_{2}=m$
and assume that $m$ is large with respect to $\varepsilon$ (we shall
see how large below). We later explain how to remove this assumption.
Choose $k=k_{2}(\varepsilon,m)$ as in Theorem \ref{thm:saturation-of-repeated-convolutions},
with $\delta=\varepsilon/2$. We shall show that the conclusion holds
if $n$ is large relative to the previous parameters.

Let $\mu,\nu\in\mathcal{P}([0,1))$. Denote 
\[
\tau=\nu^{*k}.
\]
Assuming $n$ is large enough, Theorem \ref{thm:saturation-of-repeated-convolutions}
provides us with disjoint subsets $I,J\subseteq\{0,\ldots,n\}$ with
$|I\cup J|>(1-\varepsilon/2)n$ such that 
\begin{equation}
\mathbb{P}_{i=k}\left(\tau^{x,i}\mbox{ is }(\frac{\varepsilon}{2},m)\mbox{-uniform}\right)>1-\frac{\varepsilon}{2}\qquad\mbox{for }k\in I\label{eq:55}
\end{equation}
and
\begin{equation}
\mathbb{P}_{i=k}\left(\nu^{x,i}\mbox{ is }(\varepsilon,m)\mbox{-atomic}\right)\geq1-\frac{\varepsilon}{2}\qquad\mbox{for }k\in J.\label{eq:11}
\end{equation}

Let $I_{0}\subseteq I$ denote the set of $k$ such that 
\begin{equation}
\mathbb{P}_{i=k}\left(\mu^{x,i}\mbox{ is }(\varepsilon,m)\mbox{-uniform}\right)>1-\varepsilon\qquad\mbox{for }k\in I.\label{eq:12}
\end{equation}
If $|I_{0}|>(1-\varepsilon)n$ we are done, since by \eqref{eq:11}
and \eqref{eq:12}, the pair $I_{0},J$ satisfy the second alternative
of the theorem.

Otherwise, let $I_{1}=I\setminus I_{0}$, so that $|I_{1}|=|I|-|I_{0}|>\varepsilon n/2$.
We have 
\[
\mathbb{P}_{i=k}\left(\tau^{x,i}\mbox{ is }(\frac{\varepsilon}{2},m)\mbox{-uniform and }\mu^{y,i}\mbox{ is not }(\varepsilon,m)\mbox{-uniform}\right)>\frac{\varepsilon}{2}\qquad\mbox{for }k\in I_{1}.
\]
For $\mu^{x,i},\tau^{y,i}$ in the event above, this just means that
$H_{m}(\tau^{y,i})>H_{m}(\mu^{x,i})+\varepsilon/2$ and hence $H_{m}(\mu^{x,i}*\tau^{y,i})\geq H_{m}(\mu^{x,i})+\varepsilon/2-O(1/m)$.
For any other pair $\mu^{x,i},\tau^{y,i}$ we have the trivial bound
$H_{m}(\mu^{x,i}*\tau^{y,i})\geq H_{m}(\mu^{x,i})-O(1/m)$. Thus,
using Lemmas \ref{lem:entropy-local-to-global}, \ref{lem:entropy-of-convolutions-via-component-convolutions},
\ref{lem:entropy-monotonicity-under-convolution}, 
\begin{eqnarray*}
H_{n}(\mu*\tau) & = & \mathbb{E}_{0\leq i\leq n}(H_{m}(\mu^{x,i}*\tau^{y,i}))+O(\frac{m}{n})\\
 & = & \frac{|I_{1}|}{n+1}\mathbb{E}_{i\in I_{1}}(H_{m}(\mu^{x,i}*\tau^{y,i}))+\frac{n+1-|I_{1}|}{n+1}\mathbb{E}_{i\in I_{1}^{c}}(H_{m}(\mu^{x,i}*\tau^{y,i}))+O(\frac{m}{n})\\
 & > & \frac{|I_{1}|}{n+1}\left(\mathbb{E}_{i\in I_{1}}(H_{m}(\mu^{x,i}))+(\frac{\varepsilon}{2})^{2})\right)+\frac{n+1-|I_{1}|}{n+1}\mathbb{E}_{i\in I_{1}^{c}}(H_{m}(\mu^{x,i}))+O(\frac{1}{m}+\frac{m}{n})\\
 & = & \mathbb{E}_{0\leq i\leq n}(H_{m}(\mu^{x,i}))+(\frac{\varepsilon}{2})^{3}+O(\frac{1}{m}+\frac{m}{n})\\
 & = & H_{n}(\mu)+(\frac{\varepsilon}{2})^{3}+O(\frac{1}{m}+\frac{m}{n}).
\end{eqnarray*}
So, assuming that $\varepsilon$ was sufficiently small to begin with,
$m$ large with respect to $\varepsilon$ and $n$ large with respect
to $m$, we have
\[
H_{n}(\mu*\tau)>H_{n}(\mu)+\frac{\varepsilon^{3}}{10}.
\]
On the other hand, by Proposition \ref{cor:non-discrete-KVT-theorem}
above,
\[
H_{n}(\mu*\tau)=H_{n}(\mu*\nu^{*k})\leq H_{n}(\mu)+k\cdot\left(H_{n}(\mu*\nu)-H_{n}(\mu)\right)+O(\frac{k}{n}).
\]
Assuming that $n$ is large enough in a manner depending on $\varepsilon$
and $k$, this and the previous inequality give 
\[
H_{n}(\mu*\nu)\geq H_{n}(\mu)+\frac{\varepsilon^{3}}{100k}.
\]
This is the desired conclusion, with $\delta=\varepsilon^{3}/100k$.

We now remove the largeness assumption on $m$. Let $\varepsilon,m_{1},m_{2}$
be given and choose $\varepsilon'>0$ small compared to $\varepsilon$,
and $m'$ appropriately large for $\varepsilon,m_{1},m_{2}$. Applying
what we just proved for a large enough $n$ we obtain corresponding
$I,J\subseteq[0,n]$. It will be convenient to denote $U_{1}=I$ and
$U_{2}=J$. Now, for $i\in U_{1}$, by definition of $U_{1}$ and
Lemma \ref{lem:saturation-passes-to-components}, and assuming $m_{1}/m'$
small enough,
\[
\mathbb{P}_{i\leq j\leq i+m'}(\mu^{x,j}\mbox{ is }(\sqrt{2\varepsilon'},m_{1})\mbox{-uniform})>1-\sqrt{2\varepsilon'}
\]
Thus, assuming as we may that $\varepsilon<\sqrt{2\varepsilon'}$,
if we set
\[
V_{1}=\{j\in[0,n]\,:\,\mathbb{P}_{u=j}(\mu^{x,u}\mbox{ is }(\varepsilon,m_{2})\mbox{-uniform})>1-\varepsilon\}
\]
then by Lemma \ref{lem:Chebyshev} (Chebychev's inequality), $|[i,i+m']\cap V_{1}|>(1-(2\varepsilon)^{1/4})m'$.
Similarly, defining 
\[
V_{2}=\{j\in[0,n]\,:\,\mathbb{P}_{u=j}(\mu^{x,u}\mbox{ is }(\varepsilon,m)\mbox{-atomic})>1-\varepsilon\}
\]
and using Lemma \ref{lem:almost-atomic-measures-have-almost-atomic-components},
if $m_{2}/m$ is small enough then $|[j,j+m']\cap V_{2}|>(1-(2\varepsilon)^{1/4})m'$
for all for $j\in U_{2}$. Now, applying Lemma \ref{lem:pair-of-coverings}
to $U_{1},V_{1}$ and $U_{2},V_{2}$, we find $U'_{1}\subseteq U_{1}$
and $U'_{2}\subseteq U_{2}$ as in that lemma. Taking $I'=U'_{1}$
and $J'=U'_{2}$, these are the desired sets.

Lastly, to allow for different parameters $\varepsilon_{1},\varepsilon_{2}$,
just take $\varepsilon=\min\{\varepsilon_{1},\varepsilon_{2}\}$ and
apply what we have already seen. Then any $(\varepsilon,m_{1})$-uniform
measure is $(\varepsilon_{1},m_{1})$-uniform and any $(\varepsilon,m_{2})$-atomic
measure is also $(\varepsilon_{2},m_{2})$-atomic, and we are done.
\end{proof}
Theorems \ref{thm:inverse-thm-R} and \ref{thm:self-convolution}
are formal consequences of Theorem \ref{thm:inverse-thm-Rd}, as discussed
in Section \ref{sub:inverse-theorem}.

\section{\label{sec:Parameterized-families-of-self-similar-measures}Self-similar
measures }

\subsection{\label{sub:Components-of-self-similar-measures}Uniform entropy dimension
and self-similar measures}

The entropy dimension of a measure $\theta\in\mathcal{P}(\mathbb{R})$
is the limit $\lim_{n\rightarrow\infty}H_{n}(\theta)$, assuming it
exists; by Lemma \ref{lem:entropy-local-to-global}, this i limit
is equal to $\lim_{n\rightarrow\infty}\mathbb{E}_{0\leq i\leq n}(H_{m}(\theta^{x,i}))$
for all integers $m$. The convergence of the averages does not, however,
imply that the entropies of the components $\theta^{x,i}$ concentrate
around their mean, and examples show that they need not. We introduce
the following stronger notion:
\begin{defn}
A measure $\theta\in\mathcal{P}(\mathbb{R})$ has \emph{uniform entropy
dimension} $\alpha$ if for every $\varepsilon>0$, for large enough
$m$,
\begin{equation}
\liminf_{n\rightarrow\infty}\mathbb{P}_{0\leq i\leq n}(|H_{m}(\theta^{x,i})-\alpha|<\varepsilon)>1-\varepsilon.\label{eq:uniform-e-dim}
\end{equation}

\end{defn}
Our main objective in this section is to prove:
\begin{prop}
\label{prop:component-entropy-concentration}Let $\mu\in\mathcal{P}(\mathbb{R})$
be a self-similar measure and $\alpha=\dim\mu$. Then $\mu$ has uniform
entropy dimension $\alpha$.
\end{prop}
For simplicity we first consider the case that all the contractions
in the IFS contract by the same ratio $r$. Thus, consider an IFS
$\Phi=\{\varphi_{i}\}_{i\in\Lambda}$ with $\varphi_{i}(x)=r(x-a_{i})$,
$0<r<1$. We denote the attractor by $X$ and without loss of generality
assume that $0\in X\subseteq[0,1]$, which can always be arranged
by a change of coordinates and may be seen not to affect the conclusions.
Let $\mu=\sum_{i\in\Lambda}p_{i}\cdot\varphi_{i}\mu$ be a self-similar
measure and as usual write $\varphi_{i}=\varphi_{i_{1}}\circ\ldots\circ\varphi_{i_{n}}$
and $p_{i}=p_{i_{1}}\cdot\ldots\cdot p_{i_{n}}$ for $i\in\Lambda^{n}$.

Let 
\[
\alpha=\dim\mu
\]
As we have already noted, self-similar measures are exact dimensional
\cite{FengHu09}, and for such measures the dimension and entropy
dimension coincide: 
\begin{equation}
\lim_{n\rightarrow\infty}H_{n}(\mu)=\alpha.\label{eq:14}
\end{equation}
Fix $\widetilde{x}\in X$ and define probability measures 
\[
\mu_{x,k}^{[n]}=c\cdot\sum\left\{ p_{i}\cdot\varphi_{i}\mu\,:\, i\in\Lambda^{n}\,,\,\varphi_{i}\widetilde{x}\in\mathcal{D}_{k}(x)\right\} ,
\]
where $c=c(x,\widetilde{x},k,n)$ is a normalizing constant. Thus
$\mu_{x,k}^{[n]}$ differs from $\mu_{x,k}$ in that, instead of restricting
$\mu=\sum_{i\in\Lambda^{n}}p_{i}\cdot\varphi_{i}\mu$ to $\mathcal{D}_{k}(x)$,
we include or exclude each term in its entirety depending on whether
$\varphi_{i}\widetilde{x}\in\mathcal{D}_{k}(x)$. Since $\varphi_{i}\mu$
may not be supported entirely on either $\mathcal{D}_{k}(x)$ or its
complement, in general we have neither $\mu_{x,k}^{[n]}\ll\mu_{x,k}$
nor $\mu_{x,k}\ll\mu_{x,k}^{[n]}$. Note that the definition of $\mu_{x,k}^{[n]}$
depends on the point $\widetilde{x}$, but this will not concern us.

For $0<\rho<1$ it will be convenient to write 
\[
\ell(\rho)=\left\lceil \log\rho/\log r\right\rceil ,
\]
so $\rho,r^{\ell(\rho)}$ differ by a multiplicative constant. Recall
that $\left\Vert \cdot\right\Vert $ denotes the total variation norm,
see Section \ref{sub:Preliminaries-on-entropy}. 
\begin{lem}
\label{lem:approximate-components-TV-bound}For every $\varepsilon>0$
there is a $0<\rho<1$ such that, for all $k$ and $n=\ell(\rho2^{-k})$,
\begin{equation}
\mathbb{P}_{i=k}\left(\left\Vert \mu_{x,i}-\mu_{x,i}^{[n]}\right\Vert <\varepsilon\right)>1-\varepsilon.\label{eq:13}
\end{equation}
Furthermore $\rho$ can be chosen independently of $\widetilde{x}$
and of the coordinate system on $\mathbb{R}$ (so the same bound holds
for any translate of $\mu$).\end{lem}
\begin{proof}
It is elementary that if $\mu$ is atomic then it consists of a single
atom. In this case the statement is trivial, so assume $\mu$ is non-atomic.
Then%
\footnote{This is the only part of the proof of Theorem \ref{thm:main-individual-entropy-1}
which is not effective, but with a little more work one could make
it effective in the sense that, if $\liminf-\log\Delta^{(n)}=M<\infty$,
then at arbitrarily small scales one can obtain estimates of the continuity
of $\mu$ in terms of $M$.%
} given $\varepsilon>0$ there is a $\delta>0$ such that every interval
of length $\delta$ has $\mu$-mass $<\varepsilon^{2}/2$. Choose
an integer $q$ so that $r^{q}<\delta/2$ and let $\rho=r^{q}$. 

Let $k\in\mathbb{N}$ and $\ell=\ell(2^{-k})$, so that $2^{-k}\cdot r\leq r^{\ell}\leq2^{-k}$.
Let $i\in\Lambda^{\ell}$ and consider those $j\in\Lambda^{q}$ such
that $\varphi_{ij}\mu$ is not supported on an element of $\mathcal{D}_{k}$.
Then $\varphi_{ij}\mu$ is supported on the interval $J$ of length
$\delta$ centered at one of the endpoints of an element of $\mathcal{D}_{k}$.
Since $\varphi_{i}\mu$ can give positive mass to at most two such
intervals $J$, and $\varphi_{i}\mu(J)<\varepsilon^{2}/2$ for each
such $J$, we conclude that in the representation $\mu_{i}=\frac{1}{p_{i}}\sum_{j\in\Lambda^{q}}p_{ij}\cdot(\varphi_{ij}\mu)$,
at least $1-\varepsilon^{2}$ of the mass comes from terms that are
supported entirely on just one element of $\mathcal{D}_{k}$. Therefore
the same is true in the representation $\mu=\sum_{u\in\Lambda^{\ell+q}}p_{u}\cdot\varphi_{u}\mu$.
The inequality \eqref{eq:13} now follows by an application of the
Markov inequality. Finally, Since our choice of parameters did not
depend on $\widetilde{x}$ and is invariant under translation of $\mu$
and of the IFS, the last statement holds.\end{proof}
\begin{lem}
\label{lem:component-entropy-lower-bound}For $\varepsilon>0$, for
large enough $m$ and all $k$, 
\[
\mathbb{P}_{i=k}\left(H_{m}(\mu^{x,i})>\alpha-\varepsilon\right)>1-\varepsilon,
\]
and the same holds for any translate of $\mu$.\end{lem}
\begin{proof}
Let $\varepsilon>0$ be given. Choose $0<\varepsilon'<\varepsilon$
sufficiently small that $\left\Vert \nu-\nu'\right\Vert <\varepsilon'$
implies $|H_{m}(\nu)-H_{m}(\nu')|<\varepsilon/2$ for every $m$ and
every $\nu,\nu'\in\mathcal{P}([0,1]^{d})$ (Lemma \ref{lem:entropy-total-variation-continuity}).
Let $\rho$ be as in the previous lemma chosen with respect to $\varepsilon'$.
Assume that $m$ is large enough that $|H_{m}(\mu')-\alpha|<\varepsilon/2$
whenever $\mu'$ is $\mu$ scaled by a factor of at most $\rho$ ($m$
exists by \eqref{eq:14} and Lemma \ref{lem:entropy-weak-continuity-properties}
\eqref{enu:entropy-change-of-scale}). Now fix $k$ and let $\ell=\ell(\rho2^{-k})$.
By the previous lemma and choice of $\varepsilon'$, it is enough
to show that $\frac{1}{m}H(\mu_{x,k}^{[\ell]},\mathcal{D}_{k+m})>\alpha-\varepsilon/2$.
But this follows from the fact that $\mu_{x,k}^{[\ell]}$ is a convex
combination of measures $\mu_{j}$ for $j\in\Lambda^{\ell}$, our
choice of $m$ and $\ell$, and concavity of entropy.
\end{proof}
We now prove Proposition \ref{prop:component-entropy-concentration}.
Let $0<\varepsilon<1$ be given and fix an auxiliary parameter $\varepsilon'<\varepsilon/2$.
We first show that this holds for $m$ large in a manner depending
on $\varepsilon$. Specifically let $m$ be large enough that the
previous lemma applies for the parameter $\varepsilon'$. In particular
for any $n$,
\begin{equation}
\mathbb{P}_{0\leq i\leq n}\left(H_{m}(\mu^{x,i})>\alpha-\varepsilon'\right)>1-\varepsilon'.\label{eq:60}
\end{equation}
By \eqref{eq:14}, for $n$ large enough we have $|H_{n}(\mu)-\alpha|<\varepsilon'/2$,
so by Lemma \ref{lem:entropy-local-to-global}, for large enough $n$
we have
\[
|\mathbb{E}_{0\leq i\leq n}\left(H_{m}(\mu^{x,i})\right)-\alpha|<\varepsilon'.
\]
Since $H_{m}(\mu^{x,i})\geq0$, the last two equalities imply 
\[
\mathbb{P}_{0\leq i\leq n}\left(H_{m}(\mu^{x,i})<\alpha+\varepsilon''\right)>1-\varepsilon''
\]
for some $\varepsilon''$ that tend to $0$ with $\varepsilon'$.
Thus, choosing $\varepsilon'$ small enough, the last inequality and
\eqref{eq:60} give \eqref{eq:uniform-e-dim}, as desired.

When the contraction ratios are not uniform, $\varphi_{i}=r_{i}x+a_{i}$,
some minor changes are needed in the proof. Given $n$, let $\Lambda^{(n)}$
denote the set of $i\in\Lambda^{*}=\bigcup_{m=1}^{\infty}\Lambda^{m}$
such that $r_{i}<r^{n}\leq r_{j}$, where $j$ is the same as $i$
but with the last symbol deleted (so its length is one less than $i$).
This ensures that $\{r_{i}\}_{i\in\Lambda^{(n)}}$ are all within
a multiplicative constant of each other (this constant is $\min\{r_{j}\,:\, j\in\Lambda\}$).
It is easy to check that $\Lambda^{(n)}$ is a section of $\Lambda^{*}$
in the sense that every sequence $i\in\Lambda^{*}$ with $r_{i}<r^{n}$
has a unique prefix in $\Lambda^{(n)}$. Now define $\mu_{x,k}^{[n]}$
as before, but using $\varphi_{i}\mu$ for $i\in\Lambda^{(n)}$, i.e.
\[
\mu_{x,k}^{[n]}=c\cdot\sum\left\{ p_{i}\cdot\varphi_{i}\mu\,:\, i\in\Lambda^{(n)}\,,\,\varphi_{i}\widetilde{x}\in\mathcal{D}_{k}(x)\right\} .
\]
With this modification all the previous arguments now go through. 

Finally, let us note the following consequence of the inverse theorem
(Theorem \ref{thm:inverse-thm-R}).
\begin{cor}
For every measure $\mu\in\mathcal{P}(\mathbb{R})$ with uniform entropy
dimension $0<\alpha<1$, and for every $\varepsilon>0$, there is
a $\delta>0$ and such that for all large enough $n$ and every $\nu\in\mathcal{P}([0,1])$,
\[
H_{n}(\nu)>\varepsilon\qquad\implies\qquad H_{n}(\mu*\nu)\geq H_{n}(\mu)+\delta.
\]

\end{cor}
Similar conclusions hold for dimension.

\subsection{\label{sub:Proof-of-main-thm-on-R}Proof of Theorem \ref{thm:main-individual-entropy-1} }

We again begin with the uniformly contracting case, $\varphi_{i}=rx+a_{i}$,
and continue with the notation from the previous section, in particular
assume that $0$ is in the attractor. Recall from the introduction
that 
\[
\nu^{(n)}=\sum_{i\in\Lambda^{n}}p_{i}\cdot\delta_{\varphi_{i}(0)}.
\]
Define 
\[
\tau^{(n)}(A)=\mu(r^{-n}A).
\]
One may verify easily, using the assumption $0\in X$, that 
\begin{equation}
\mu=\nu^{(n)}*\tau^{(n)}.\label{eq:scale-n-convolution}
\end{equation}
As in the introduction, write 
\[
n'=[n\log(1/r)].
\]
Thus $\tau^{(n)}$ is $\mu$ scaled down by a factor of $r^{n}=2^{-n'}$
and translated. Using \eqref{eq:14}, Lemma \ref{lem:entropy-weak-continuity-properties},
and the fact that $\tau^{(n)}$ is supported on an interval of order
$r^{n}=2^{-n'}$, we have
\[
\lim_{n\rightarrow\infty}\frac{1}{n'}H(\nu^{(n)},\mathcal{D}_{n'})=\lim_{n\rightarrow\infty}\frac{1}{n'}H(\mu,\mathcal{D}_{n'})=\dim\mu=\alpha.
\]

Suppose now that $\alpha<1$. Fix a large $q$ and consider the identity
\begin{eqnarray*}
\frac{1}{qn}H(\mu,\mathcal{D}_{qn}) & = & \frac{n'}{qn}\cdot\left(\frac{1}{n'}H(\mu,\mathcal{D}_{n'})\right)+\frac{qn-n'}{qn}\cdot\left(\frac{1}{qn-n'}H(\mu,\mathcal{D}_{qn}|\mathcal{D}_{n'})\right)\\
 & = & \frac{[\log(1/r)]}{q}\left(\frac{1}{n'}H(\mu,\mathcal{D}_{n'})\right)+\frac{q-[\log(1/r)]}{q}\left(\frac{1}{qn-n'}H(\mu,\mathcal{D}_{qn}|\mathcal{D}_{n'})\right).
\end{eqnarray*}
The left hand side and the term $\frac{1}{n'}H(\mu,\mathcal{D}_{n'})$
on the right hand side both tend to $\alpha$ as $n\rightarrow\infty$.
Since $r,q$ are independent of $n$ we conclude that
\begin{equation}
\lim_{n\rightarrow\infty}\frac{1}{qn-n'}H(\mu,\mathcal{D}_{qn}|\mathcal{D}_{n'})=\alpha.\label{eq:61}
\end{equation}
From the identity $=\mathbb{E}_{i=n'}(\nu_{y,i}^{(n)})$ and linearity
of convolution, 
\[
\mu=\nu^{(n)}*\tau^{(n)}=\mathbb{E}_{i=n'}\left(\nu_{y,i}^{(n)}*\tau^{(n)}\right).
\]
Also, each measure $\nu_{y,i}^{(n)}*\tau^{(n)}$ is supported on an
interval of length $O(2^{-n'})$ so 
\[
|H(\nu_{y,i}^{(n)}*\tau^{(n)},\mathcal{D}_{qn}|\mathcal{D}_{n'})-H(\nu_{y,i}^{(n)}*\tau^{(n)},\mathcal{D}_{qn})|=O(1).
\]
By concavity of conditional entropy (Lemma \ref{lem:entropy-combinatorial-properties}
\eqref{enu:entropy-concavity}), 
\begin{eqnarray*}
H(\mu,\mathcal{D}_{qn}|\mathcal{D}_{n'}) & = & H(\nu^{(n)}*\tau^{(n)},\mathcal{D}_{qn}|\mathcal{D}_{n'})\\
 & \geq & \mathbb{E}_{i=n'}\left(H(\nu_{y,i}^{(n)}*\tau^{(n)},\mathcal{D}_{qn}|\mathcal{D}_{n'})\right)\\
 & = & \mathbb{E}_{i=n'}\left(H(\nu_{y,i}^{(n)}*\tau^{(n)},\mathcal{D}_{qn})\right)+O(1),
\end{eqnarray*}
so by \eqref{eq:61},
\begin{equation}
\limsup_{n\rightarrow\infty}\frac{1}{qn-n'}\mathbb{E}_{i=n'}\left(H(\nu_{y,i}^{(n)}*\tau^{(n)},\mathcal{D}_{qn})\right)\leq\alpha.\label{eq:29}
\end{equation}
Now, we also know that
\begin{equation}
\lim_{n\rightarrow\infty}\frac{1}{qn-n'}H(\tau^{(n)},\mathcal{D}_{qn})=\alpha,\label{eq:62}
\end{equation}
since, up to a re-scaling, this is just \eqref{eq:14} (we again used
the fact that $\tau^{(n)}$ is supported on intervals of length $2^{-n'}$).
By Lemma \ref{lem:entropy-monotonicity-under-convolution}, for every
component $\nu_{y,i}^{(n)}$, 
\[
\frac{1}{qn-n'}H(\nu_{y,i}^{(n)}*\tau^{(n)},\mathcal{D}_{qn})\geq\frac{1}{qn-n'}H(\tau^{(n)},\mathcal{D}_{qn})+O(\frac{1}{qn-n'}).
\]
Therefore for every $\delta>0$, 
\[
\lim_{n\rightarrow\infty}\mathbb{P}_{i=n'}\left(\frac{1}{qn-n'}H(\nu_{y,i}^{(n)}*\tau^{(n)},\mathcal{D}_{qn})>\alpha-\delta\right)=1
\]
which, combined with \eqref{eq:29}, implies that for every $\delta>0$,
\[
\lim_{n\rightarrow\infty}\mathbb{P}_{i=n'}\left(\left|\frac{1}{qn-n'}H(\nu_{y,i}^{(n)}*\tau^{(n)},\mathcal{D}_{qn})-\alpha\right|<\delta\right)=1,
\]
and replacing $\alpha$ with the limit in \eqref{eq:62}, we have
that for all $\delta>0$, 
\begin{equation}
\lim_{n\rightarrow\infty}\mathbb{P}_{i=n'}\left(\left|\frac{1}{qn-n'}H(\nu_{y,i}^{(n)}*\tau^{(n)},\mathcal{D}_{qn})-\frac{1}{qn-n'}H(\tau^{(n)},\mathcal{D}_{qn})\right|<\delta\right)=1.\label{eq:59}
\end{equation}

Now let $\varepsilon>0$. By Proposition \ref{prop:component-entropy-concentration}
and the assumption that $\alpha<1$, for small enough $\varepsilon$,
large enough $m$ and all sufficiently large $n$,
\begin{eqnarray*}
\mathbb{P}_{n'<i\leq qn'}\left(H_{m}((\tau^{(n)})^{x,i})<1-\varepsilon\right) & \geq & \mathbb{P}_{n'<i\leq qn'}\left(H_{m}((\tau^{(n)})^{x,i})<\alpha+\varepsilon\right).\\
 & > & 1-\varepsilon
\end{eqnarray*}
Choose $\delta>0$ smaller than the constant of the same name in the
conclusion of Theorem \ref{thm:inverse-thm-R}. Then, for sufficiently
large $n$, we can apply Theorem \ref{thm:inverse-thm-R} to the components
$\nu_{y,i}^{(n)}$ in the event in equation \eqref{eq:59} (for this
we re-scale by $2^{n'}$ and note that the measures $\nu_{y,n'}^{(n)}$
are supported on level-$n'$ dyadic cells and $\tau^{(n)}$ is supported
on an interval of the same order of magnitude). We conclude that every
component $\nu_{y,i}^{(n)}$ in the event in question satisfies $\frac{1}{qn-n'}H(\nu_{y,i}^{(n)},\mathcal{D}_{qn})<\varepsilon$,
and hence by \eqref{eq:59}, 
\[
\lim_{n\rightarrow\infty}\mathbb{P}_{i=n'}\left(\frac{1}{qn-n'}H(\nu_{y,i}^{(n)},\mathcal{D}_{qn})<\varepsilon\right)=1.
\]
Thus, from the definition of conditional entropy and the last equation,
\begin{eqnarray*}
\lim_{n\rightarrow\infty}\frac{1}{qn-n'}H(\nu^{(n)},\mathcal{D}_{qn}|\mathcal{D}_{n'}) & = & \lim_{n\rightarrow\infty}\frac{1}{qn-n'}\mathbb{E}_{i=n'}\left(H(\nu_{y,i}^{(n)},\mathcal{D}_{qn})\right)\\
 & = & \lim_{n\rightarrow\infty}\mathbb{E}_{i=n'}\left(\frac{1}{qn-n'}H(\nu_{y,i}^{(n)},\mathcal{D}_{qn})\right)\\
 & < & \varepsilon.
\end{eqnarray*}
Since $\varepsilon$ was arbitrary, this is Theorem \ref{thm:main-individual-entropy-1}.

\subsection{\label{sub:Inverse-thm-proof-non-uniformly-contracting}Proof of
Theorem \ref{thm:main-individual-entropy-1-1} (the non-uniformly
contracting case)}

We now consider the situation for general IFS, in which the contraction
$r_{i}$ of $\varphi_{i}$ is not constant. Again assume that $0$
is in the attractor. Let $r=\prod_{i\in\Lambda}r_{i}^{p_{i}}$, $n'=\log_{2}(1/r)$
as in the introduction, and define $\widetilde{\nu}^{(n)}$ as before.
Given $n$, let 
\[
R_{n}=\{r_{i}\,:\, i\in\Lambda^{n}\}.
\]
Note that $|R_{n}|=O(n^{|\Lambda|})$. Therefore $H(\widetilde{\nu}^{(n)},\{\mathbb{R}\}\times\mathcal{F})=O(\log n)$,
and consequently for all $k$ 
\[
H(\widetilde{\nu}^{(n)},\widetilde{\mathcal{D}}_{k})=H(\nu^{(n)},\mathcal{D}_{k})+O(\log n).
\]
Thus 
\[
H(\widetilde{\nu}^{(n)},\widetilde{\mathcal{D}}_{qn}|\widetilde{\mathcal{D}}_{n'})=H(\nu^{(n)},\mathcal{D}_{qn}|\mathcal{D}_{n})+O(\log n),
\]
and our goal reduces to proving that for every $q>1$,
\[
\frac{1}{qn}H(\nu^{(n)},\mathcal{D}_{qn}|\mathcal{D}_{n'})\rightarrow0\qquad\mbox{as }n\rightarrow\infty.
\]
Furthermore, for every $\varepsilon>0$
\[
H(\nu^{(n)},\mathcal{D}_{qn}|\mathcal{D}_{(1-\varepsilon)n'})=H(\nu^{(n)},\mathcal{D}_{qn}|\mathcal{D}_{n'})-O(\varepsilon n),
\]
so it will suffice for us to prove that
\[
\limsup_{n\rightarrow\infty}\frac{1}{qn}H(\nu^{(n)},\mathcal{D}_{qn}|\mathcal{D}_{(1-\varepsilon)n})=o(1)\qquad\mbox{as }\varepsilon\rightarrow0.
\]

Fix $\varepsilon>0$. For $t\in R_{n}$ let 
\begin{eqnarray*}
\Lambda^{n,t} & = & \{i\in\Lambda^{n}\,:\, r_{i}=t\}\\
p^{n,t} & = & \sum_{i\in\Lambda^{n,t}}p_{i},
\end{eqnarray*}
so $\{p^{n,t}\}_{t\in R_{n}}$ is a probability vector. It will sometimes
be convenient to consider $i\in\Lambda^{n}$, $i\in\Lambda^{n,t}$
and $t\in R_{n}$ as random elements drawn according to the probabilities
$p_{i}$, $p_{i}/p^{n,t}$, and $p^{n,t}$, respectively. Then we
interpret expressions such as $\mathbb{P}_{i\in\Lambda^{n}}(A)$,
$\mathbb{P}_{i\in\Lambda^{n,t}}(A)$ and $\mathbb{P}_{t\in R_{n}}(A)$
in the obvious manner, and similarly expectations. With this notation,
we can define 
\begin{eqnarray*}
\nu^{(n,t)} & = & \mathbb{E}_{i\in\Lambda^{n,t}}(\delta_{\varphi_{i}(0)})=\frac{1}{p^{n,t}}\sum_{i\in\Lambda^{n,t}}p_{i}\cdot\delta_{\varphi_{i}(0)}.
\end{eqnarray*}
This a probability measure on $\mathbb{R}$ representing the part
of $\nu^{(n)}$ coming from contractions by $t$; indeed, 
\begin{eqnarray}
\nu^{(n)} & = & \mathbb{E}_{t\in R_{n}}(\nu^{(n,t)}).\label{eq:57}
\end{eqnarray}
For $t>0$ let $\tau^{(t)}$ be the measure
\[
\tau^{(t)}(A)=\tau(tA)
\]
(note that we are no longer using logarithmic scale, so the measure
that was previously denoted$\tau^{(n)}$ is now $\tau^{(2^{-n})}$).
We then have
\begin{equation}
\mu=\mathbb{E}_{t\in R_{n}}(\nu^{(n,t)}*\tau^{(t)}).\label{eq:56}
\end{equation}

Fix $\varepsilon>0$. Arguing as in the previous section, using equation
\eqref{eq:56} and concavity of entropy, we have 
\begin{eqnarray}
\alpha & = & \lim_{n\rightarrow\infty}\frac{1}{qn-(1-\varepsilon)n'}H(\mu,\mathcal{D}_{qn}|\mathcal{D}_{(1-\varepsilon)n'})\nonumber \\
 & \geq & \limsup_{n\rightarrow\infty}\frac{1}{qn-(1-\varepsilon)n'}\mathbb{E}_{t\in R_{n}}\left(H(\nu^{(n,t)}*\tau^{(t)},\mathcal{D}_{qn}|\mathcal{D}_{(1-\varepsilon)n'})\right).\label{eq:58}
\end{eqnarray}
By the law of large numbers, 
\[
\lim_{n\rightarrow\infty}\mathbb{P}_{i\in\Lambda^{n}}\left(2^{-(1+\varepsilon)n'}<r_{i}<2^{-(1-\varepsilon)n'}\right)=1,
\]
or, equivalently,
\begin{equation}
\lim_{n\rightarrow\infty}\mathbb{P}_{t\in R_{n}}\left(2^{-(1+\varepsilon)n'}<t<2^{-(1-\varepsilon)n'}\right)=1.\label{eq:63}
\end{equation}
Using $H_{k}(\mu)\rightarrow\alpha$ and the definition of $\tau^{(t)}$,
we conclude that 
\[
\lim_{n\rightarrow\infty}\mathbb{P}_{t\in R_{n}}\left(\frac{1}{qn-(1-\varepsilon)n'}H(\tau^{(t)},\mathcal{D}_{qn})\geq(1-\varepsilon)\alpha\right)=1.
\]
Also, since $\tau^{(t)}$ is supported on an interval of order $t$,
from \eqref{eq:63}, \eqref{eq:58} and concavity of entropy, 
\begin{eqnarray}
\alpha & \geq & \limsup_{n\rightarrow\infty}\frac{1}{qn-(1-\varepsilon)n'}\mathbb{E}_{t\in R_{n}}\mathbb{E}_{i=n'}\left(H(\nu_{y,i}^{(n,t)}*\tau^{(t)},\mathcal{D}_{qn}|\mathcal{D}_{(1-\varepsilon)n'})\right)\nonumber \\
 & = & \limsup_{n\rightarrow\infty}\frac{1}{qn-(1-\varepsilon)n'}\mathbb{E}_{t\in R_{n}}\mathbb{E}_{i=n'}\left(H(\nu_{y,i}^{(n,t)}*\tau^{(t)},\mathcal{D}_{qn}\right).\label{eq:64}
\end{eqnarray}
This is the analogue of Equation \ref{eq:29} in the proof of the
uniformly contracting case and from here one proceeds exactly as in
that proof to conclude that there is a function $\delta(\varepsilon)$,
tending to $0$ as $\varepsilon\rightarrow0$, such that 
\[
\mathbb{P}_{t\in R_{n}}\left(\mathbb{P}_{i=n'}\left(\frac{1}{qn-(1-\varepsilon)n}H(\nu^{(n,t)},\mathcal{D}_{qn})<\delta(\varepsilon)\right)\right)=1.
\]
Now, using Equation \eqref{eq:57} and the fact that the entropy of
the distribution $\{p^{(n,t)}\}_{t\in R_{n}}$ is $o(n)$ as $n\rightarrow\infty$,
by Lemma \ref{lem:entropy-combinatorial-properties} \eqref{enu:entropy-almost-convexity}
one concludes that
\[
\limsup_{n\rightarrow\infty}H(\nu^{(n)},\mathcal{D}_{qn}|\mathcal{D}_{(1-\varepsilon)n'})\leq\delta(\varepsilon),
\]
which is what we wanted to prove.

\subsection{\label{sub:Transversality-and-exceptions}Transversality and the
dimension of exceptions}

In this section we prove Theorem \ref{thm:main-parametric}. Let $I\subseteq\mathbb{R}$
be a compact interval for $t\in I$ and let $\Phi_{t}=\{\varphi_{i,t}\}_{i\in\Lambda}$
be an IFS, $\varphi_{i,t}(x)=r_{i}(t)(x-a_{i}(t))$. We define $\varphi_{i,t}$
and $r_{i}(t)$ for $i\in\Lambda^{n}$ as usual, set $\Delta_{i,j}(t)=\varphi_{i,t}(0)-\varphi_{j,t}(0)$
when $i,j\in\Lambda^{n}$ and for $i,j\in\Lambda^{\mathbb{N}}$ define
$\Delta_{i,j}(t)=\lim\Delta_{i_{1}\ldots i_{n},j_{1}\ldots j_{n}}(t)$
(this is well defined since $\lim\varphi_{i_{1}\ldots i_{n}}(0)$
converges, in fact exponentially, as $n\rightarrow\infty$).

For $i,j\in\Lambda^{n}$ or $i,j\in\Lambda^{\mathbb{N}}$ let $i\land j$
denote the longest common initial segment of $i,j$, and $|i\land j|$
its length, so $|i\land j|=\min\{k\,:\, i_{k}\neq j_{k}\}-1$. Let
\[
r_{min}=\min_{i\in\Lambda}\min_{t\in I}|r_{i}(t)|,
\]
so $0<r_{min}<1$. For a $C^{k}$-function $F:I\rightarrow\mathbb{R}$
write $F^{(p)}=\frac{d^{p}}{dt^{p}}F$, and 
\[
\left\Vert F\right\Vert _{I,k}=\max_{p\in\{0,\ldots,k\}}\max_{t\in I}|F^{(p)}(t)|.
\]
In particular we write 
\[
R_{k}=\max_{i\in\Lambda}\left\Vert r_{i}\right\Vert _{I,k}.
\]

\begin{defn}
The family $\{\Phi_{t}\}_{t\in I}$ is \emph{transverse of order $k$
}if $r_{i}(\cdot),a_{i}(\cdot)$ are $k$-times continuously differentiable
and there is a constant $c>0$ such that for every $n\in\mathbb{N}$
and distinct $i,j\in\Lambda^{n}$, 
\begin{equation}
\forall\; t_{0}\in I\quad\exists\; p\in\{0,1,2,\ldots,k\}\quad\mbox{such that}\quad|\Delta_{i,j}^{(p)}(t_{0})|\geq c\cdot|i\land j|^{-p}\cdot r_{i\land j}(t_{0}).\label{eq:transversality}
\end{equation}

\end{defn}
The classical notion of transversality roughly corresponds to the
case $k=1$ in this definition, see e.g. \cite[Definition 2.7]{PeresSchlag2000}.
Unlike the classical notion, which either fails or is difficult to
verify in many cases of interest, higher-order transversality holds
almost automatically. To begin with, let $i,j\in\Lambda^{n}$ and
observe that 
\begin{eqnarray*}
\Delta_{i,j}(t) & = & r_{i\land j}(t)\widetilde{\Delta}_{i,j}(t),
\end{eqnarray*}
where, writing $u,v$ for the sequences obtained from $i,j$ after
deleting the longest initial segment, 
\begin{eqnarray*}
\widetilde{\Delta}_{i,j}(t) & = & \Delta_{u,v}(t).
\end{eqnarray*}
Differentiating $p$ times, 
\begin{eqnarray*}
\widetilde{\Delta}_{i,j}^{(p)}(t) & = & \frac{d^{p}}{dt^{p}}(r_{i\land j}(t)^{-1}\cdot\Delta_{i,j}(t))\\
 & = & \sum_{q=0}^{p}\binom{p}{q}\cdot\frac{d^{q}}{dt^{q}}(r_{i\land j}(t)^{-1})\cdot\Delta_{i,j}^{(p-q)}(t).
\end{eqnarray*}
A calculation shows that 
\[
|\frac{d^{q}}{dt^{q}}(r_{i\land j}(t)^{-1})|\leq O_{q,r_{min},R_{q}}(|i\land j|^{q}\cdot r_{i\land j}(t)^{-1}).
\]
Thus we have the bound
\[
|\widetilde{\Delta}_{i,j}^{(p)}(t)|=O_{p,r_{min},,R_{p}}\left(\max_{0\leq q\leq p}\left(|i\land j|^{q}\cdot r_{i\land j}(t)^{-1}\cdot|\Delta_{i,j}^{(q)}(t)|\right)\right).
\]

\begin{prop}
\label{prop:analitic-implies-transverse}Suppose $r_{i}(\cdot),a_{i}(\cdot)$
are real-analytic on $I$. Suppose that for $i,j\in\Lambda^{\mathbb{N}}$,
$\Delta_{i,j}\equiv0$ on $I$ if and only if $i=j$. Then the associated
family $\{\Phi_{t}\}_{t\in I}$ is transverse of order $k$ for some
$k$.\end{prop}
\begin{proof}
First, for $x\in I$ we can extend $r_{i},a_{i}$ analytically to
a complex neighborhood $U_{x}$ of $x$ on which $|r_{i}|$ are still
bounded uniformly away from $1$. Define $\Delta_{i,j}(z)$ as before
for $i,j\in\Lambda^{n}$ and $z\in U_{x}$, and note that for $i,j\in\Lambda^{\mathbb{N}}$
the limit $\Delta_{i,j}(z)=\lim\Delta_{i_{1}\ldots i_{n},j_{1}\ldots j_{n}}(z)$
is uniform for $z\in U_{x}$. This shows that $\Delta_{i,j}(t)$ is
also real-analytic on $I$

Given $k$, from the expression for $\widetilde{\Delta}_{i,j}^{(p)}$
above, we see that if $c>0$ and there exists $t_{0}\in I$ such that
$|\Delta_{i,j}^{(p)}(t_{0})|\leq c\cdot|i\land j|^{-p}\cdot r_{i\land j}(t_{0})$
for all $0\leq p\leq k$, then $|\widetilde{\Delta}_{i,j}^{(p)}(t_{0})|\leq c'$
for all $0\leq p\leq k$, where $c'=O_{k,R_{k}}(c)$. For each $k$
choose $c_{k}>0$ such that the associated $c'_{k}$ satisfies $c'_{k}<1/k$. 

Suppose that for all $k$ the family $\{\Phi_{t}\}$ is not transverse
of order $k$. Then by assumption we can choose $n(k)$ and distinct
$i^{(k)},j^{(k)}\in\Lambda^{n(k)}$, and a point $t_{k}\in I$, such
that $|\Delta_{i^{(k)},j^{(k)}}^{(p)}(t_{k})|\leq c_{k}\cdot|i^{(k)}\land j^{(k)}|^{-p}\cdot r_{i^{(k)}\land j^{(k)}}(t_{k})$
for $0\leq p\leq k$, and hence $\widetilde{\Delta}_{i^{(k)},j^{(k)}}^{(p)}(t_{k})\leq c'_{k}$.
Let $u^{(k)}$ and $v^{(k)}$ denote the sequences obtained from $i^{(k)}$
and $j^{(k)}$ by deleting the first $|i^{(k)}\land j^{(k)}|$ symbols,
so that the first symbols of $u^{(k)}$ and $v^{(k)}$ now differ
and $\Delta_{u^{(k)},v^{(k)}}=\widetilde{\Delta}_{i^{(k)},j^{(k)}}$.
Hence we have 
\begin{equation}
|\Delta_{u^{(k)},v^{(k)}}^{(p)}(t_{k})|\leq c'_{k}<1/k\quad\mbox{for all }\quad0\leq p\leq k.\label{eq:53}
\end{equation}

Passing to a subsequence $k_{\ell}$, we may assume that $t_{k_{\ell}}\rightarrow t_{0}$
and that $u^{(k_{\ell})}\rightarrow u\in\Lambda^{\mathbb{N}}$ and
$v^{(k_{\ell})}\rightarrow v\in\Lambda^{\mathbb{N}}$ (the latter
in the sense that all coordinates stabilize eventually to the corresponding
coordinate in the limit sequence). Note that $u\neq v$, because $u^{(k_{\ell})},v^{(k_{\ell})}$
differ in their first symbol for all $\ell$, hence so do $u,v$.
It follows that $\Delta_{u^{(k_{\ell})},v^{(k_{\ell})}}\rightarrow\Delta_{u,v}$
uniformly and that the same holds for $p$-th derivatives. Hence for
all $p\geq0$, using uniform convergence and \eqref{eq:53}, 
\[
|\Delta_{u,v}^{(p)}(t_{0})|=\lim_{\ell\rightarrow\infty}|\Delta_{u^{(k_{\ell})},v^{(k_{\ell})}}^{(p)}(t_{k_{\ell}})|=0.
\]
But $\Delta_{u,v}$ is real analytic so the vanishing of its derivatives
implies $\Delta_{u,v}\equiv0$ on $I$, contrary to the hypothesis.
\end{proof}
We turn now to the implications of transversality. The key implication
is provided by the following simple lemma. 
\begin{lem}
\label{lem:transversality}Let $k\in\mathbb{N}$ and let $F$ be a
$k$-times continuously differentiable function on a compact interval
$J\subseteq\mathbb{R}$. Let $M=\left\Vert F\right\Vert _{J,k}$ and
let $0<c<1$ be such that for every $x\in J$ there is a $p\in\{0,\ldots,k\}$
with $|F^{(p)}(x)|>c$. Then for every $0<\rho<c/2^{k}$, the set
$F^{-1}(-\rho,\rho)\subseteq J$ can be covered by $O_{k,M,|J|}(1/c^{2})$
intervals of length $\leq2(\rho/c)^{1/2^{k}}$ each. \end{lem}
\begin{proof}
For brevity, we shall suppress dependence on the parameters $k,M,|J|$,
so throughout this proof, $O(\cdot)=O_{k,M,|J|}(\cdot)$. 

The proof is by induction on $k$. For $k=0$ the hypothesis is that
$|F^{(0)}(x)|=|F(x)|>c$ for all $x\in J$, hence $F^{-1}(-\rho,\rho)=\emptyset$
for $0<\rho<c=c/2^{0}$, and the assertion is trivial.

Assume that we have proved the claim for $k-1$ and consider the case
$k$. Let $J'$ be a maximal closed interval in $F^{-1}[-c,c]$ and
let $G=F'|_{J'}$. Note that $G$ satisfies the hypothesis for $k-1$
and the same value of $c$ and $M$, and $\sqrt{c\rho}<c/2^{k-1}$,
so from the induction hypothesis we find that $G^{-1}(-\sqrt{c\rho},\sqrt{c\rho})$
can be covered by $O(1/c)$ intervals of length $<2(\sqrt{c\rho}/c)^{1/2^{k-1}}=2(\rho/c)^{1/2^{k}}$
each. Let $U$ denote the union of this cover and consider the intervals
$J'_{i}$ which are the closures of the maximal sub-intervals in $J'\setminus U$.
By the above, the number of such intervals $J'_{i}$ is $\leq O(1/c)$.
Now, on each $J_{i}'$ we have $|F'|\geq\sqrt{c\rho}$, so by continuity
of $F'$ either $F'\geq\sqrt{c\rho}$ or $F'\leq-\sqrt{c\rho}$ in
all of $J_{i}'$. An elementary consequence of this is that $J'_{i}\cap F^{-1}(-\rho,\rho)$
is an interval of length at most $2\rho/\sqrt{c\rho}=2\sqrt{\rho/c}\leq2(\rho/c)^{1/2^{k}}$.
In summary we have covered $J'\cap F^{-1}(-\rho,\rho)$ by $O(1/c)$
intervals of length $2(\rho/c)^{1/2^{k}}$ each.

It remains to show that there are $O(1/c)$ maximal intervals $J'\subseteq F^{-1}[-c,c]$
as in the paragraph above. In fact, we only need to bound the number
of such $J'$ that intersect $F^{-1}(-\rho,\rho)$. For $J'$ of this
kind, if $J'=J$ we are done, since this means there is just one such
interval. Otherwise there is an endpoint $a\in J'$ with $|F(a)|=c$.
There is also a point $b\in J'$ with $|F(b)|<\rho<c/2^{k}$. Since
$|F'|\leq M$, we conclude that $|J'|\geq|b-a|\geq(c-\rho)/M\geq c/2M$.
Thus, since the intervals $J'$ are disjoint, their number is $\leq|J|/(c/2M)=O(1/c)$,
completing the induction step.
\end{proof}
Let $\bdim X$ denote the upper box dimension of a set $X$, defined
by 
\[
\bdim X=\limsup_{r\rightarrow0}\frac{\log\#\min\{\ell\,:\, X\mbox{ can be covered by }\ell\mbox{ balls of radius }r\}}{\log(1/r)}.
\]
One always has $\dim X\leq\bdim X$. The packing dimension is defined
by 
\[
\pdim X=\inf\{\sup_{n}\bdim X_{n}\,:\, X\subseteq\bigcup_{n=1}^{\infty}X_{n}\}.
\]
Note that $\dim X\leq\pdim X$, and $Y\subseteq X$ implies $\pdim Y\leq\pdim X$.
\begin{thm}
\label{thm:transverse-implies-small-exceptions}If $\{\Phi_{t}\}_{t\in I}$
satisfies  transversality of order $k\geq1$ on the compact interval
$I$, then\textup{\emph{ the set $E$ of ``exceptional'' parameters
in Theorem \ref{thm:description-of-exceptional-params} has}} packing
(and hence Hausdorff) dimension $0$. \end{thm}
\begin{proof}
Write 
\[
M=\sup_{n}\sup_{i,j\in\Lambda^{n}}\left\Vert \Delta_{i,j}\right\Vert _{I,k}.
\]
That $M<\infty$ follows from $k$-fold continuous differentiability
of $r_{i}(\cdot),a_{i}(\cdot)$ and the fact that $|r_{i}|$ are bounded
away from $1$ on $I$. By transversality there is a constant $c>0$
such that for every $t\in I$, every $n$ and all distinct $i,j\in\Lambda^{n}$,
\[
|\frac{\partial^{p}}{\partial t^{p}}\Delta_{i,j}(t)|>c\cdot|i\land j|^{-p}\cdot r_{min}^{|i\land j|}\qquad\mbox{for some }p\in\{0,\ldots,k\}.
\]
In what follows we suppress the dependence on $k,M,c$ and $|I|$
in the $O(\cdot)$ notation: $O(\cdot)=O_{k,M,c,|I|}(\cdot)$.

We may assume that $c<1$ and $k\geq2$. Let $\varepsilon<cr_{min}/2k$
and fix $n$ and distinct $i,j\in\Lambda^{n}$. By the previous lemma,
for all $0<\rho<c|i\land j|^{-k}r_{min}^{|i\land j|}/2^{k}$, and
in particular for $0<\rho<cr_{min}^{n}/(2n)^{k}$, the set $\{t\in I\,:\,|\Delta_{i,j}|<\rho\}$
can be covered by at most $O((2n)^{k}/r_{min}^{n})$ intervals of
length $2((2n)^{k}\rho/r_{min}^{n})^{1/2^{k}}$ each. Now set $\rho=\varepsilon^{n}$
(our choice of $\varepsilon$ guarantees that $\rho$ is in the proper
range) and let $i,j$ range over their $\leq|\Lambda|^{n}$ different
possible values. We find that the set

\[
E_{\varepsilon,n}=\bigcup_{i,j\in\Lambda^{n}\,,\, i\neq j}(\Delta_{i,j})^{-1}(-\varepsilon^{n},\varepsilon^{n})
\]
can be covered by $O((2n)^{k}|\Lambda|^{n}/r_{min}^{n})$ intervals
of length $\leq((2n)^{k}\varepsilon^{n}/r_{min}^{n})^{1/2^{k}}$.
Now, $E\subseteq E_{\varepsilon}$ where 
\begin{equation}
E_{\varepsilon}=\bigcup_{N=1}^{\infty}\bigcap_{n>N}E_{\varepsilon,n}.\label{eq:46}
\end{equation}
By the above, for each $\varepsilon$ and $N$ we have 
\begin{eqnarray*}
\bdim\left(\bigcap_{n>N}E_{\varepsilon,n}\right) & \leq & \lim_{n\rightarrow\infty}\frac{\log\left(O(2n)^{k}|\Lambda|^{n}/r_{min}^{n}\right)}{\log\left(((2n)^{k}\varepsilon^{n}/r_{min}^{n})^{1/2^{k}}\right)}\\
 & = & O(2^{k}\frac{\log(|\Lambda|/r_{min})}{\log(\varepsilon/r_{min})}).
\end{eqnarray*}
The last expression is $o(1)$ as $\varepsilon\rightarrow0$, uniformly
in $N$. Thus by \eqref{eq:46}, the same is true of $E_{\varepsilon}$,
and $E\subseteq E_{\varepsilon}$ for all $\varepsilon$, so $E$
has packing (and Hausdorff) dimension $0$.
\end{proof}
Theorem \ref{thm:main-parametric} now follows by combining Proposition
\ref{prop:analitic-implies-transverse} and Theorem \ref{thm:transverse-implies-small-exceptions}.

\subsection{\label{sub:Applications}Miscellaneous proofs}

To complete the proof of Corollary \ref{cor:algebraic-parameters}
we have:
\begin{lem}
\label{lem:Liouville-bound}Let $A\subseteq\mathbb{R}$ be a finite
set of algebraic numbers over $\mathbb{Q}$. Then there is a constant
$0<s<1$ such that any polynomial expression $x$ of degree $n$ in
the elements of $A$, either $x=0$ or $|x|>s^{n}$. \end{lem}
\begin{proof}
Choose an algebraic integer $\alpha$ such that $A\subseteq\mathbb{Q}(\alpha)$.
Since the statement is unchanged if we multiply all elements of $A$
by an integer, we can assume that the elements of $A$ are integer
polynomials in $\alpha$ of degree $\leq d$ and coefficients bounded
by $N$, for some $d,N$. Substituting these polynomials into the
expression for $x$, we have an expression $x=\sum_{k=0}^{dn}n_{k}\alpha^{k}$
where $n_{k}\in\mathbb{N}$ and $|n_{k}|\leq N$. It suffices to prove
that any such expression is either $0$ or $\geq s^{n}$ for $0<s<1$
independent of $n$ (but which may depend on $\alpha$ and hence on
$d,N$). In proving this last statement we may assume that $d=1$
(replace $s$ by $s^{1/d}$ and change variables to $n'=dn$).

Let $\alpha=\alpha_{1},\alpha_{2},\ldots,\alpha_{d}$ denote the algebraic
conjugates of $\alpha$ and $\sigma_{1},\sigma_{2},\ldots,\sigma_{d}$
the automorphisms of $\mathbb{Q}(\alpha)$, with $\sigma_{i}\alpha=\alpha_{i}$.
If $x\neq0$ then $\prod_{i=1}^{d}\sigma_{i}(x)\in\mathbb{Z}$, so
\[
1\leq|\prod_{i=1}^{d}\sigma_{i}(x)|=x\cdot\prod_{i=2}^{d}|\sum_{k=0}^{n}n_{k}\sigma_{i}(x)^{k}|\leq x\cdot\prod_{i=2}^{d}\sum_{k=0}^{n}n_{k}|\alpha_{i}|^{k}\leq x\cdot(n\cdot N\cdot\alpha_{\max}^{n})^{d},
\]
where $\alpha_{\max}=\max\{|\alpha_{2}|,\ldots,|\alpha_{d}|\}$. Dividing
out gives the lemma.
\end{proof}
We finish with some comments on Sinai's problem, Theorem \ref{thm:Sinais-problem}.
We first state a generalization of Theorem \ref{thm:description-of-exceptional-params}
needed to treat families of IFSs that contract only on average.

Suppose that for $t\in I$ we have a family $\Phi_{t}=\{\varphi_{i,t}\}_{i\in\Lambda}$
of (not necessarily contracting) similarities of $\mathbb{R}$, and
as usual write $\varphi_{i,t}=r_{i,t}U_{i,t}+a_{i,t}$. Let $p$ be
a fixed probability vector and suppose that for each $t$ we have
$\sum p_{i}'\log r_{i}<0$, i.e. the systems contract on average.
One can then show that there is a unique probability measure $\mu_{t}$
on $\mathbb{R}$ satisfying $\mu_{t}=\sum_{i\in\Lambda}p_{i}\cdot\varphi_{i,t}\mu_{t}$
\cite{NicolSidorovBroomhead2002}, that $H(\mu_{t},\mathcal{D}_{m})<\infty$
for every $t$ and $m$, and that $\mu_{t}([-R,R])\rightarrow1$ as
$R\rightarrow\infty$ uniformly in $t$. Under these conditions one
can verify the stronger property that for every $t\in I$ we have
\[
\left|H_{m}(\mu_{t})-H_{m}((\mu_{t})_{[-R,R]})\right|=o(1)\qquad\mbox{as }R\rightarrow\infty
\]
uniformly in $t$ and $m$. 
\begin{thm}
Let $(\Phi_{t})_{t\in I}$, $p$, and $\mu_{t}$ be as in the preceding
paragraph. Let $\widetilde{\mu}$ denote the product measure on $\Lambda^{\mathbb{N}}$
with marginal $p$, and suppose that $A\subseteq\Lambda^{\mathbb{N}}$
is a Borel set such that $\widetilde{\mu}(A)>0$. Write 
\[
E=\bigcap_{\varepsilon>0}\left(\bigcup_{N=1}^{\infty}\,\bigcap_{n>N}\left(\bigcup_{i,j\in A}(\Delta_{i,j})^{-1}((-\varepsilon^{n},\varepsilon^{n}))\right)\right).
\]
Then $\dim\mu_{t}=\min\{d,\sdim\mu_{t}\}$ for every $t\in I\setminus E$.
Furthermore suppose that $I\subseteq\mathbb{R}$ is compact and connected,
and that the parametrization is analytic in the sense of Theorem \ref{thm:main-parametric}.
If 
\[
\forall i,j\in A\quad\left(\;\Delta_{i,j}\equiv0\mbox{ on }I\quad\iff\quad i=j\;\right)
\]
then the set $E$ above is of packing (and Hausdorff) dimension at
most $k-1$, and in particular of Lebesgue measure $0$.
\end{thm}
The proof is the same as the proofs of Theorems \ref{thm:description-of-exceptional-params}
and \ref{thm:main-parametric}, except that in analyzing the resulting
convolution one must approximate $\mu_{t}$ by $(\mu_{t})_{[-R,R]}$
for an appropriately large $R$ that is fixed in advance, with the
scale $n$ large relative to $R$. We omit the details.

Let us see how this applies to Theorem \ref{thm:Sinais-problem},
where $\varphi_{-1,\alpha}(x)=(1-\alpha)x-1$ and $\varphi_{1,\alpha}(x)=(1+\alpha)x+1$
for $\alpha\in(0,1]$, and $p=(1/2,1/2)$. It suffices to consider
the system for $\alpha\in[s,1]$ for some $s>0$. Let $A$ be the
set of $i\in\Lambda^{\mathbb{N}}$ such that $|\frac{1}{N}\sum_{n=1}^{N}i_{n}-\frac{1}{2}|<\delta$
for $n>N(\delta)$, where $\delta>0$ small enough to ensure that
$|\varphi_{i_{1}\ldots i_{n}}|<1$ when this condition holds, and
$N(\delta)$ large enough that $\widetilde{\mu}(A)>0$; in fact we
can make $\widetilde{\mu}(A)$ arbitrarily close to $1$, by the law
of large numbers. It remains to verify for $i,j\in A$ that $\Delta_{i,j}$
vanishes on $[s,1]$ if and only if $i=j$. Note that for $i\in\{-1,1\}^{n}$,
\[
\varphi_{i,\alpha}(0)=1+(1+i_{1}\alpha)+(1+i_{1}\alpha)(1+i_{2}\alpha)+\ldots+\prod_{k=1}^{n}(1+i_{k}\alpha).
\]
Thus $\Delta_{i,j}$ is a series whose terms are of the form $c_{k,m}(1-\alpha)^{k}(1+\alpha)^{m}$
for some $c_{k,m}\in\{0,\pm1\}$, and $i=j$ if and only if all terms
are $0$. Furthermore, there is an $n_{0}$ such that if $k+m\geq n_{0}$
and $c_{k,m}\neq0$, then $k>(1-\delta)m$. Thus since $s\leq\alpha\leq1$
and $\delta$ was chosen small enough, the series converges uniformly
on $[s,1]$, and furthermore there is ~an $\varepsilon>0$ such that
the series converges uniformly on some larger interval $[s,1+\varepsilon]$,
and even in a neighborhood of $1$ in the complex plane. Hence $\Delta_{i,j}(\cdot)$
is real-analytic on $[s,1+\varepsilon]$ and is given by this series.
Now, if $i\neq j$ we can divide out by the highest power $(1-\alpha)^{k_{0}}$
that is common to all the terms (possibly $k_{0}=0$), and evaluate
the resulting function at $\alpha=1$. We get a finite sum of the
form $\sum_{(k,m)\in U}c_{m,k}2^{m}$ for some finite set of indices
$U\in\mathbb{N}^{2}$ such that $c_{m,k}\in\{\pm1\}$ for $(k,m)\in U$.
Such a sum cannot vanish, hence by analyticity $\Delta_{i,j}\not\equiv0$
on every sub-interval of $[s,1+\varepsilon]$, and in particular $\Delta_{i,j}\not\equiv0$
on $[s,1]$, as desired.

{% empty space...

}

\bibliographystyle{plain}
\bibliography{bib}

\bigskip{}

\lyxaddress{Email: mhochman@math.huji.ac.il \\
Address: Einstein Institute of Mathematics, Givat Ram, Jerusalem 91904,
Israel}
\end{document}